\documentclass[12pt]{amsart}

\usepackage{latexsym}
\usepackage{psfrag}
\usepackage{amsmath}
\usepackage{amssymb}
\usepackage{epsfig}
\usepackage{amsfonts}
\usepackage{amscd}
\usepackage{mathrsfs}
\usepackage{graphicx}
\usepackage{enumerate}
\usepackage[autostyle=false, style=english]{csquotes}
\MakeOuterQuote{"}
\usepackage{ragged2e}
\usepackage[all]{xy}
\usepackage{mathtools}
\newlength\ubwidth

\usepackage[dvipsnames]{xcolor}

\usepackage{placeins}

\newtheorem{theorem}{Theorem}
\newtheorem{lemma}[theorem]{Lemma}

\newtheorem{ex}[theorem]{Example}

\newtheorem{cor}[theorem]{Corollary}
\newtheorem{prop}[theorem]{Proposition}

\newtheorem{defn}[theorem]{Definition}

\begin{document}

\title[Constructing Generating Functions]{ CONSTRUCTION OF EVIDENTLY POSITIVE SERIES AND AN ALTERNATIVE CONSTRUCTION FOR A FAMILY OF PARTITION GENERATING FUNCTIONS DUE TO KANADE AND RUSSELL}

\author{Kağan Kurşungöz}
\address{Kağan Kurşungöz, Faculty of Engineering and Natural Sciences, Sabanc{\i} University, Tuzla, Istanbul 34956, Turkey}
\email{kursungoz@sabanciuniv.edu}

\author{Hal\.{ı}me Ömrüuzun Seyrek}
\address{Hal\.{ı}me Ömrüuzun Seyrek, Faculty of Engineering and Natural Sciences, Sabanc{\i} University, Tuzla, Istanbul 34956, Turkey}
\email{halimeomruuzun@sabanciuniv.edu}

	
\date{June 2021}
	
\begin{abstract}
\noindent We give an alternative construction for a family of partition generating functions due to Kanade and Russell. In our alternative construction, we use ordinary partitions instead of jagged partitions. We also present new generating functions which are evidently positive series for  partitions due to Kanade and Russell. To obtain those generating functions, we first construct an evidently positive series for a key infinite product. In that construction, a series of combinatorial moves is used to decompose an arbitrary partition into a base partition together with some auxiliary partitions that bijectively record the moves. 
	 
\end{abstract}
	
\maketitle
	
	
\section{Introduction and Statement of Results}
\label{secIntro}

A \textit{partition} of a positive integer $n$ is a finite non-decreasing sequence of positive integers $\lambda=\lambda_1$, \ldots, $\lambda_k$ such that  $\lambda_1+ \ldots +\lambda_k=n$. The $\lambda_i$'s are called the \textit{parts} of the partition and the \textit{weight} $|\lambda|$ of the partition $\lambda$ is defined to be $|\lambda|=n$.  We consider the empty sequence as the only partition of zero. For example, the five partitions of $n=4$ are
\begin{center}
	$4$, \quad $1+3$, \quad  $2+2$, \quad $1+1+2$, \quad $1+1+1+1$.
\end{center}

Depending on context, we sometimes allow zeros to appear in the partition. Clearly, the zeros have no contribution to the weight of the partition, but the length changes as we add or take out zeros.

Many partition identities have the following form: "the number of partitions of $n$ satisfying condition $A$ = the number of partitions of $n$ satisfying condition $B$" \cite{the-theory-of-partitions}. We recall the famous partition identity due to Euler \cite{euler}:

\begin{theorem}
The number of partitions of $n$ into distinct parts is equal to the number of partitions of $n$ into odd parts. 
\end{theorem} 
The \textit{generating function} $f(q)$ for the sequence $a_0, a_1, a_2, a_3, \ldots$ is the power series $f(q)=\sum_{n\geq 0}a_nq^n$ \cite{the-theory-of-partitions}. Any form of the series 
\begin{align*}
	F(q) = \sum_{n \geq 0} p(n \vert \textrm{ cond. } A) q^n
\end{align*}
is called a \emph{partition generating function}.  

Here and throughout, for $n \geq 0$,

\begin{center}
	$(a; q)_n=(1-a)(1-aq)(1-aq^2)\ldots (1-aq^n)$\\
	\vspace{3mm}
	$(a; q)_\infty=\lim_{n \to \infty}(a; q)_n=(1-a)(1-aq)(1-aq^2)(1-aq^3)\ldots.$
\end{center}
and we assume $|q| < 1$ for convergence. 
We shall frequently use the following $q$-series identities due to Euler:
\begin{equation}\label{Euler 1}
	\frac{1}{(x; q)_\infty}= \sum_{n \geq 0} \frac{x^n}{(q; q)_{n}},
\end{equation}

\begin{equation}\label{Euler 2}
	(-x; q)_\infty= \sum_{n \geq 0} \frac{x^nq^{n(n-1)/2}}{(q; q)_{n}}.
\end{equation}

In 2019, Kanade and Russell provided analytic sum-sides to some partition identities of Rogers-Ramanujan type \cite{sk-mcr-staircases}. In Section \ref{Alternative Construction}, we give an alternative construction of those analytic sum-sides. In our alternative construction, we use ordinary partitions instead of jagged partitions.

\begin{theorem} \cite{sk-mcr-staircases}
	 \label{2+2 not allowed}
Consider the partitions satisfying the following conditions:\begin{enumerate} [($a$)]
\item  No consecutive parts allowed. 
\item  Odd parts do not repeat.
\item  For a contiguous sub-partition $\lambda_i+\lambda_{i+1}+\lambda_{i+2}$, we have $| \lambda_i-\lambda_{i+2}| \geq 4 $ if $\lambda_{i+1}$ is even and appears more than once.
\item $2+2$ is not allowed as a sub-partition.
\end{enumerate}
	
\noindent For $n, m \in \mathbb{N}$, let $kr_1(n, m)$ denote the number of partitions of n into m parts such that the partitions satisfy the conditions (a), (b), (c) and (d). Then,
\begin{equation} \label{first generating function}
\sum_{m,n \geq 0}kr_1(n,m)q^nt^m= \sum_{i, j, k \geq 0}(-1)^k \frac{t^{i+2j+3k}q^{(i+2j+3k)(i+2j+3k-1)+i+6j+3k^2+6k}}{(q;q)_i(q^4;q^4)_j(q^6;q^6)_k}. 
\end{equation}
\end{theorem}
 
We also prove two theorems which are variants of Theorem \ref{2+2 not allowed} in the sense that we change the condition $(d)$ with different initial conditions. 

In Section \ref{Evidently Positive Series}, we construct evidently positive series as the generating functions of the partitions in the theorems presented in Section \ref{Alternative Construction}. To obtain those evidently positive series, we first construct an evidently positive series for a key infinite product. In that construction, a series of combinatorial moves is used to decompose an arbitrary partition into a base partition together with some auxiliary partitions that bijectively record the moves. We state some of our main results below.

\begin{theorem} \label{each part appears at most twice}
	For $n, m \in \mathbb{N}$, let $h(n, m)$ denote the number of partitions of $n$ into $m$ parts such that each part appears at most twice. Then,
	\begin{equation} \label{each part appears at most twice-generating function}
		\sum_{m, n \geq 0}h(n, m)q^nt^m=\sum_{\substack{n_1, n_2 \geq 0 \\ n_1=n_{11}+n_{12} \\ \beta}}\frac{q^{| \beta |}t^{2n_2+n_1}}{(q; q)_{n_{12}}(q^3; q^3)_{n_2}} 
	\end{equation} 
	where $\beta$ is the base partition with $n_2$ pairs,  $n_{11}$ immobile singletons, $n_{12}$ moveable singletons.
\end{theorem}

The definitions of \textit{base partition}, \textit{immobile singleton} and \textit{moveable singleton } will be presented in the proof of Theorem \ref{each part appears at most twice} and it is necessary to present the definitions in the flow of the proof. 
\begin{lemma} \label{polynomials}
	For $m_1, m_2, m_3 \in \mathbb{N}$ and $m \in \mathbb{Z}$, let $P(m_1,m_2, m_3, m+1; q)$ be the generating function of the base partitions $\beta$'s defined in the proof of Theorem \ref{each part appears at most twice} with $m_1$ pairs of two repeating parts, $m_2$ pairs of two consecutive parts and $m_3$ blocks, where a block is a partition into five parts which have the form $[k-1, k], k, [k+2, k+2]$. Then, 
	\begin{align} \label{P=P0+P1}
		P(m_1,m_2, m_3, m+1; q)=&P_0(m_1,m_2, m_3, m+1; q)\nonumber\\&+P_1(m_1,m_2, m_3, m+1; q)
	\end{align}
	where $P_0(m_1,m_2, m_3, m+1; q)$ is the generating function of the base partitions in which the largest pair is $[m, m]$ and $P_1(m_1,m_2, m_3,m+1; q)$ is the generating function of the base partitions in which the largest pair is $[m, m+1]$.  $P_0(m_1,m_2, m_3,m+1; q)$ and $P_1(m_1,m_2, m_3, m+1; q)$ satisfy the following functional equations: 
	\begin{align} \label{functional equation of P0}
		P_0(m_1,m_2, m_3, m+1; q)&=q^{2m}\Big[P_0(m_{1}-1,m_2, m_3, m; q) \nonumber \\&+P_1(m_{1}-1,m_2, m_3, m-1; q) \nonumber
		\\&+P_0(m_{1}-1,m_2, m_3, m-1; q)\Big] \nonumber
		\\&+q^{5m-7}\Big[P_1(m_1,m_2, m_3-1, m-3; q) \nonumber 
		\\&+P_0(m_1,m_2, m_3-1, m-3; q) \nonumber
		\\&+P_1(m_1,m_2, m_3-1, m-4; q)\Big] \\
		P_1(m_1,m_2, m_3, m+1; q)&=q^{2m+1}\Big[P_1(m_1,m_2-1, m_3, m; q) \nonumber \\&+P_0(m_1,m_2-1, m_3, m; q) \nonumber
		\\&+P_1(m_1,m_2-1, m_3, m-1; q)\Big] \label{functional equation of P1}
	\end{align}
	\begin{align} 
		&\begin{rcases} \label{initial conditions}
			P_{0/1}(m_1,m_2, m_3, m; q)&=0      \quad \quad \text{if $m <  0$}  \\
			P_{0/1}(m_1,m_2, m_3, 0; q)&=1 \\
			P_{0/1}(0,0, 0, m; q)&=0      \quad \quad \text{if $m \neq 1$} \\
			P_0(0, 0, 0, 1; q)&=1  \\
			P_1(0, 0, 0, 1; q)&=0
		\end{rcases}
		\text{initial conditions}
	\end{align}
	Moreover, $P(m_1,m_2, m_3, m+1; q)$'s are the polynomials of $q$ with evidently positive coefficients. 
\end{lemma}

\begin{theorem} \label{2+2 not allowed-NEW}
	Consider the partitions satisfying the following conditions:
	
	\begin{enumerate} [($a$)]
		\item  No consecutive parts allowed. 
		\item  Odd parts do not repeat.
		\item  For a contiguous sub-partition $\lambda_i+\lambda_{i+1}+\lambda_{i+2}$, we have $| \lambda_i-\lambda_{i+2} | \geq 4 $ if $\lambda_{i+1}$ is even and appears more than once.
		\item $2+2$ is not allowed as a sub-partition.
	\end{enumerate}
	
	\noindent For $n, m \in \mathbb{N}$, let $kr_1(n, m)$ denote the number of partitions of n into m parts such that the partitions satisfy the conditions (a), (b), (c) and (d). Then,
	\begin{align} \label{first new generating function}
		&\sum_{m,n \geq 0}kr_1(n,m)q^nt^m \nonumber\\
		&= \sum_{\substack{m_1, m_2, m_3, \\ m,n_{12}, i, j, k\geq 0}} \frac{P(m_1, m_2, m_3, m+1; q^2)q^{2mn_{12}+2n_{12}^2+i+4j+(2m_1+2m_2+5m_3+n_{12}+i+2j+k)^2}}{(q^2;q^2)_{n_{12}}(q^6;q^6)_{m_1+m_2+2m_3}(q^2;q^2)_i(q^4;q^4)_j} \nonumber \\
		&\times t^{2m_1+2m_2+5m_3+n_{12}+i+2j+k} 
	\end{align}
	where $P(m_1, m_2, m_3, m+1; q)$'s are the polynomials of $q$  with evidently positive coefficients constructed in Lemma \ref{polynomials}. Moreover, the generating function (\ref{first new generating function}) is an evidently positive series. 
\end{theorem}

In Section \ref{secSpecialCases}, we list a few closed formulas 
for the polynomials $P(m_1,m_2, m_3, m+1; q)$ in Lemma \ref{polynomials} in some special cases.  
We also list some of the $P(m_1,m_2, m_3, m+1; q)$'s for small values of the parameters.  

We conclude with some commentary and an open problem in Section \ref{Conclusion}.

\section{An Alternative Construction for a Family of Partition Generating Functions due to Kanade and Rusell}
\label{Alternative Construction}	

\begin{defn} \normalfont
	Let $\lambda=\lambda_1+ \ldots+ \lambda_m$ be a partition counted by $kr_1(n, m)$. If there exist repeating even parts $(2k)+(2k)$ in $\lambda$, then we rewrite those parts as consecutive odd parts $(2k-1)+(2k+1)$. We call the partition formed after this transformation a "\emph{seed partition}".
\end{defn}

\begin{prop}\label{evens appearing even times}
	The ordinary partitions in which $0$ may appear as a part is generated by 	
	\begin{equation*}
		A(t;q;a)=\prod_{n = 1}^{\infty} \frac{(1+tq^{2n}+ (a-1)t^2q^{4n})}{(1-tq^{2n-1})(1-t^2q^{4n})}.\frac{1}{(1-t)},
	\end{equation*}
	where the exponent of $t$ keeps track of the number of parts and the exponent of $a$  keeps  track of the number of non-zero even parts that appear an even number of times.
\end{prop}

\begin{proof}

We will keep track of the number of parts by the exponent of $t$, and keep track of the number of non-zero even parts that appear an even number of times by the exponent of $a$. Since we don't put any restrictions on parts, we will use all parts in all possible ways when we write the generating function. Let 
\begin{align*}
	A(t;q;a)=&(1+tq^0+t^2q^0+t^3q^0+t^4q^0+t^5q^0+ \ldots ).\\
	&(1+tq+t^2q^2+t^3q^3+t^4q^4+t^5q^5+ \ldots ).\\
	& (1+tq^2+at^2q^4+t^3q^6+at^4q^8+t^5q^{10}+ \ldots ).\\
	& (1+tq^3+t^2q^6+t^3q^9+t^4q^{12}+t^5q^{15}+ \ldots ).\\
	& (1+tq^4+at^2q^8+t^3q^{12}+at^4q^{16}+t^5q^{20}+ \ldots ).\\
	& (1+tq^5+t^2q^{10}+t^3q^{15}+t^4q^{20}+t^5q^{25}+ \ldots ).\\
	& \vdots 
\end{align*}

\noindent The first infinite sum is for $0$'s, the second infinite sum is for $1$'s, etc. For example, the term $at^4q^8$ means that we use four $2$'s, and since $2$ appears an even number of times in this case, we have a factor of $a$.

\noindent Using geometric series and elementary factorization identities \cite{GA-number theory}, we obtain:
\begin{align*}
	A(t;q;a)=&\frac{1}{(1-t)}.\frac{1}{(1-tq)}.\frac{(1+tq^2+(a-1)t^2q^4)}{(1-t^2q^4)}.\\ &.\frac{1}{(1-tq^3)}.\frac{(1+tq^4+(a-1)t^2q^8)}{(1-t^2q^8)}. \ldots\\
	=&\prod_{n = 1}^{\infty} \frac{(1+tq^{2n}+ (a-1)t^2q^{4n})}{(1-tq^{2n-1})(1-t^2q^{4n})}.\frac{1}{(1-t)}.
\end{align*}
\end{proof}

\noindent \emph{\textbf{Proof of Theorem \ref{2+2 not allowed}:}} Let $\lambda=\lambda_1+ \ldots +\lambda_m$ be a partition counted by $kr_1(n, m)$ and $\overline{\lambda}=\overline{\lambda_1}+ \ldots +\overline{\lambda_m}$ be the corresponding seed partition. 
\noindent We define $\beta=1+3+5+ \ldots +(2m-1)$ as the "base partition"  with $m$ consecutive odd parts. Let $\mu= \mu_1+ \ldots +\mu_m$, where $\mu_i=\overline{\lambda_i}-\beta_i$ for all $i=1, \ldots, m$.
\noindent Whenever the seed partition $\overline{\lambda}$ has a streak of consecutive odd parts, the number of parts in that streak must be even and that streak will give rise to a group of the same repeating even parts in $\mu$. Assume that $\overline{\lambda}=3+5+7+10$. Then, by the definition of the seed partition, we expect that $\lambda=4+4+7+10$ or $\lambda=3+6+6+10$. But, in both cases, the condition $(c)$ is violated. From this example, we see that the number of consecutive odd parts in a streak of $\overline{\lambda}$ must be even.

\noindent The generating function for ordinary partitions where we keep track of the number of non-zero even parts that appear an even number of times is given in Proposition \ref{evens appearing even times}.

\noindent We recall that the exponent of $a$ keeps track of the number of non-zero even parts that appear an even number of times, and from a seed partition we can generate
\begin{center}
	$2^{\text{the number of non-zero even parts appearing an even number of times in $\mu$} }$
\end{center} 
partitions satisfying the conditions, together with the seed partition itself. Therefore, we take $a=2$ in $A(t;q;a)$ and we get the following infinite product:

\begin{equation*} 
	A(t;q;2)=\prod_{n= 1}^{\infty} \frac{(1+tq^{2n}+ t^2q^{4n})}{(1-tq^{2n-1})(1-t^2q^{4n})}.\frac{1}{(1-t)}.
\end{equation*}

\noindent After some straightforward calculations, we obtain:
\begin{align}	 \label{without weight of base partition}
	&\prod_{n= 1}^{\infty} \frac{(1+tq^{2n}+ t^2q^{4n})}{(1-tq^{2n-1})(1-t^2q^{4n})}.\frac{1}{(1-t)}=\prod_{n = 1}^{\infty} \frac{(1-t^3q^{6n})}{(1-tq^{2n-1})(1-tq^{2n})(1-t^2q^{4n})(1-t)}  \nonumber \\
	&=\frac{(t^3q^6; q^6)_\infty}{(t; q)_\infty(t^2q^4; q^4)_\infty} \nonumber \\
	&= \left( \sum_{i \geq 0}\frac{t^i}{(q; q)_i} \right) \left( \sum_{j \geq 0}\frac{t^{2j}q^{4j}}{(q^4; q^4)_j} \right)  \left( \sum_{k \geq 0}\frac{(-1)^kt^{3k}q^{6k+3k(k-1)}}{(q^6; q^6)_k} \right) \nonumber \\
	&=\sum_{i, j, k \geq 0} \frac{(-1)^kt^{i+2j+3k}q^{4j+3k^2+3k}}{(q; q)_i(q^4; q^4)_j(q^6; q^6)_k}.
\end{align}

\noindent For the third identity, we use the identities (\ref{Euler 1}) and (\ref{Euler 2}). The remaining work is adding the weight of base partition to the exponent of $q$ in (\ref{without weight of base partition}).  Consider the base partition $\beta=1+3+5+ \ldots +(2m-1)$. The partition $\beta$ has $m$ parts and its weight is $m^2$. Let $m=i+2j+3k$. The weight of the base partition $\beta$ becomes $(i+2j+3k)^2$. Therefore, the generating function of partitions satisfying the conditions $(a)$, $(b)$, $(c)$ and $(d)$ is
\begin{align*} 
	\sum_{m,n \geq 0}kr_1(n,m)q^nt^m&=\sum_{i, j, k \geq 0} \frac{(-1)^kt^{i+2j+3k}q^{4j+3k^2+3k}\color{red}{q^{(i+2j+3k)^2}}}{(q; q)_i(q^4; q^4)_j(q^6; q^6)_k} \nonumber \\	
	&=\sum_{i, j, k \geq 0}(-1)^k \frac{t^{i+2j+3k}q^{(i+2j+3k)(i+2j+3k-1)+i+6j+3k^2+6k}}{(q;q)_i(q^4;q^4)_j(q^6;q^6)_k}. 
\end{align*}

\noindent The latter generating function is exactly the sum (3.2) which is constructed by Kanade and Russell in \cite{sk-mcr-staircases} (Where they used $x$ for $t$).
\qed

\begin{ex} \normalfont
	Let $\lambda=3+5+8+11+13+19+21+23+25$. In this case, the seed partition is $\overline{\lambda}=3+5+8+11+13+19+21+23+25$ and the base partition is $\beta=1+3+5+7+9+11+13+15+17$. Then, we get $\mu= 2+2+3+4+4+8+8+8+8$.  The sum $2+2$ corresponds to  $3+5$ in $\overline{\lambda}$ and it corresponds to $1+3$ in $\beta$. The sum $4+4$ corresponds to  $11+13$ in $\overline{\lambda}$ and it corresponds to $7+9$ in $\beta$. The sum $8+8+8+8$ correspond to $19+21+23+25$ in $\overline{\lambda}$ and it corresponds to $11+13+15+17$ in $\beta$. Moreover, the number of non-zero even parts that appear an even number of times in $\mu$ determines the number of partitions that can be generated from the seed partition $\overline{\lambda}$. In this example, we have three non-zero even parts that appear an even number of times in $\mu$, namely $2$, $4$ and $8$. So, from the seed partition $\overline{\lambda}$, we can generate $2^3=8$ new partitions (together with $\overline{\lambda}$) counted by $kr_1(n, m)$, where $n=128$, $m=9$. Let us list all of these partitions:
	
	\begin{center}
		$3+5+8+11+13+19+21+23+25=\overline{\lambda}$\\
		$\color{red}{4+4}$ $+$ $8+11+13+19+21+23+25$\\
		$3+5+8$ $+$ $\color{green}{12+12}$ $+$ $19+21+23+25$\\
		$\color{red}{4+4}$ $+$ $8$ $+$ $\color{green}{12+12}$ $+$ $19+21+23+25$\\
		$3+5+8+11+13$ $+$ $\color{blue}{20+20+24+24}$\\
		$\color{red}{4+4}$ $+$ $8+11+13$ $+$ $\color{blue}{20+20+24+24}$\\
		$3+5+8$ $+$ $\color{green}{12+12}$ $+$ $\color{blue}{20+20+24+24}$\\
		$\color{red}{4+4}$ $+$ $8$ $+$ $\color{green}{12+12}$ $+$ $\color{blue}{20+20+24+24}$.
	\end{center}
	
\end{ex}

\begin{theorem} \cite{sk-mcr-staircases} \label{1 not allowed}
	Consider the partitions satisfying the following conditions:
	
	\begin{enumerate}
		\item [(a)] No consecutive parts allowed. 
		\item [(b)] Odd parts do not repeat.
		\item [(c)] For a contiguous sub-partition $\lambda_i+\lambda_{i+1}+\lambda_{i+2}$, we have $| \lambda_i-\lambda_{i+2} | \geq 4 $ if $\lambda_{i+1}$ is even and appears more than once.
		\item [(d')] $1$ is not allowed to appear as a part.
	\end{enumerate}
	
	\noindent For $n, m \in \mathbb{N}$, let $kr_2(n, m)$ denote the number of partitions of n into m parts such that the partitions satisfy the conditions (a), (b), (c) and (d'). Then,
	\begin{equation} \label{second generating function}
		\sum_{m,n \geq 0}kr_2(n,m)q^nt^m= \sum_{i, j, k \geq 0}(-1)^k \frac{t^{i+2j+3k}q^{(i+2j+3k)(i+2j+3k-1)+2i+2j+3k^2+6k}}{(q;q)_i(q^4;q^4)_j(q^6;q^6)_k}.
	\end{equation}  
\end{theorem}

Before the proof of Theorem \ref{1 not allowed}, we give a definition and prove a proposition. 
\begin{defn} \normalfont \label{almost-seed partition}
	Let $\lambda$ be a partition counted by $kr_2(n, m)$. If there exist repeating even parts $(2k)+(2k)$ in $\lambda$, then we rewrite those parts as consecutive odd parts $(2k-1)+(2k+1)$. We call the partition formed after this transformation an "\emph{almost-seed partition}".
\end{defn}
\noindent If the partition $\lambda$ has the sub-partition $2+2$, the partition which is constructed by the method in Definition \ref{almost-seed partition} is a partition satisfying all above conditions $(a), (b), (c)$ except the last condition $(d')$. After the transformation $2+2$ becomes $1+3$ which is ruled out. That is why we use the term "almost-seed" in our definition. 

\begin{prop}\label{evens appearing even times-with even number of 0s}
	The ordinary partitions in which $0$ may appear as a part, and if it appears as a part, it must appear an even number of times is generated by 
	
	\begin{equation*}
		B(t; q; a)=\prod_{n= 1}^{\infty} \frac{(1+tq^{2n}+ (a-1)t^2q^{4n})}{(1-tq^{2n-1})(1-t^2q^{4n})}.\frac{1}{(1-t^2)},
	\end{equation*}	
	where the exponent of $t$ keeps track of the number of parts and the exponent of $a$  keeps  track of the number of non-zero even parts that appear an even number of times. 
\end{prop}

\noindent 
\begin{proof}
We will keep track of the number of parts by the exponent of $t$, and keep track of the number of non-zero even parts that appear an even number of times by the exponent of $a$. Since we don't put any restrictions on parts, we will use all parts in all possible ways when we write the generating function. Let
\begin{align*}
	B(t; q; a)=&(1+t^2q^0+t^4q^0+t^6q^0+t^8q^0+ \ldots ).\\
	&(1+tq+t^2q^2+t^3q^3+t^4q^4+t^5q^5+ \ldots ).\\
	& (1+tq^2+at^2q^4+t^3q^6+at^4q^8+t^5q^{10}+ \ldots ).\\
	& (1+tq^3+t^2q^6+t^3q^9+t^4q^{12}+t^5q^{15}+ \ldots ).\\
	& (1+tq^4+at^2q^8+t^3q^{12}+at^4q^{16}+t^5q^{20}+ \ldots ).\\
	& (1+tq^5+t^2q^{10}+t^3q^{15}+t^4q^{20}+t^5q^{25}+ \ldots ).\\
	& \vdots 
\end{align*}

\noindent The first infinite sum is for $0$'s, the second infinite sum is for $1$'s, etc. For example, the term $at^4q^8$ means that we use four $2$'s, and since $2$ appears an even number of times in this case, we have a factor of $a$.

\noindent Using geometric series and elementary factorization identities, we obtain:
\begin{align*}
	B(t;q;a)=&\frac{1}{(1-t^2)}.\frac{1}{(1-tq)}.\frac{(1+tq^2+(a-1)t^2q^4)}{(1-t^2q^4)}.\\
	&.\frac{1}{(1-tq^3)}.\frac{(1+tq^4+(a-1)t^2q^8)}{(1-t^2q^8)}. \ldots\\
	=&\prod_{n = 1}^{\infty} \frac{(1+tq^{2n}+ (a-1)t^2q^{4n})}{(1-tq^{2n-1})(1-t^2q^{4n})}.\frac{1}{(1-t^2)}.
\end{align*}

\end{proof}

\noindent \emph{\textbf{Proof of Theorem \ref{1 not allowed}:}} We follow the idea in the proof of Theorem \ref{2+2 not allowed}. Let $\lambda=\lambda_1+ \ldots+ \lambda_m$ be a partition counted by $kr_2(n, m)$ and $\overline{\lambda}=\overline{\lambda_1}+ \ldots+ \overline{\lambda_m}$ be the corresponding almost-seed partition. We define $\beta=1+3+5+ \ldots+ (2m-1)$ as the "base partition" with $m$ consecutive odd parts. Let $\mu= \mu_1+ \ldots+ \mu_m$, where $\mu_i=\overline{\lambda_i}-\beta_i$ for all $i=1, \ldots, m$. Whenever the almost-seed partition $\overline{\lambda}$ has a streak of consecutive odd parts, this streak will give rise to a group of the same repeating even parts in $\mu$. We also note that if the partition $\mu$ consists of a streak of $0$'s at the beginning, then the number of $0$'s in this streak has to be an even number. This fact follows from the difference condition $(c)$.  If an almost-seed partition starts with $1+3$, then during the generation of all partitions satisfying the conditions, we change $1+3$ as $2+2$ due to the condition $(d')$. We also observe that if an almost seed partition starts with a streak of consecutive odd parts, then this streak must have an even number of parts. To the contrary, assume that $\overline{\lambda}$ starts  with an odd number of consecutive parts, set $\overline{\lambda}=1+3+5+10$. By definition of the almost seed partition, the corresponding partition $\lambda$ to $\overline{\lambda}$ can be $2+2+5+10$ or $1+4+4+10$. The partition $2+2+5+10$ violates the condition $(c)$, the partition $1+4+4+10$ violates the conditions $(c)$ and $(d')$. This also explains why the number of $0$'s in a streak of $\mu$ must be even, if $\mu$ starts with $0$'s.

\noindent The generating function for ordinary partitions where we keep track of the number of non-zero even parts that appear an even number of times and if $0$ appears as a part then it must appear an even number of times is given in Proposition \ref{evens appearing even times-with even number of 0s}.

\noindent We recall that the exponent of $a$ keeps track of the number of non-zero even parts that appear an even number of times, and from an almost-seed partition we can generate
\begin{center}
	$2^{\text{the number of non-zero even parts appearing an even number of times in $\mu$} }$
\end{center} 
partitions satisfying the conditions. Therefore, we take $a=2$ in $B(t; q; a)$ and we get the following infinite product:

\begin{equation*}
	B(t; q; 2)=	\prod_{n = 1}^{\infty} \frac{(1+tq^{2n}+ t^2q^{4n})}{(1-tq^{2n-1})(1-t^2q^{4n})}.\frac{1}{(1-t^2)}.
\end{equation*}

\noindent After some straightforward calculations, we obtain:
\begin{align} \label{without weight of base partition-2}
	&\prod_{n = 1}^{\infty} \frac{(1+tq^{2n}+ t^2q^{4n})}{(1-tq^{2n-1})(1-t^2q^{4n})}.\frac{1}{(1-t^2)}	
	=\prod_{n= 1}^{\infty} \frac{(1-t^3q^{6n})}{(1-tq^{2n-1})(1-tq^{2n})(1-t^2q^{4n})(1-t^2)} \nonumber \\
	&=\frac{(t^3q^6; q^6)_\infty}{(tq; q)_\infty(t^2; q^4)_\infty} \nonumber \\
	&= \left( \sum_{i \geq 0}\frac{t^iq^i}{(q; q)_i} \right) \left( \sum_{j \geq 0}\frac{t^{2j}}{(q^4; q^4)_j} \right)  \left( \sum_{k \geq 0}\frac{(-1)^kt^{3k}q^{6k+3k(k-1)}}{(q^6; q^6)_k} \right) \nonumber \\
	&=\sum_{i, j, k \geq 0} \frac{(-1)^kt^{i+2j+3k}q^{i+3k^2+3k}}{(q; q)_i(q^4; q^4)_j(q^6; q^6)_k}.
\end{align}

\noindent For the third identity, we use the identities (\ref{Euler 1}) and (\ref{Euler 2}). The remaining work is adding the weight of the base partition to the exponent of $q$ in (\ref{without weight of base partition-2}).  Consider the base partition $\beta=1+3+5+\ldots+ (2m-1)$. The partition $\beta$ has $m$ parts and its weight is $m^2$. Let $m=i+2j+3k$. Then the weight of the base partition $\beta$ is $(i+2j+3k)^2$. Therefore, the generating function of partitions satisfying the conditions $(a)$, $(b)$, $(c)$ and $(d')$ is
\begin{align*} 
	\sum_{m,n \geq 0}kr_2(n,m)q^nt^m &=\sum_{i, j, k \geq 0} \frac{(-1)^kt^{i+2j+3k}q^{i+3k^2+3k}\color{red}{q^{(i+2j+3k)^2}}}{(q; q)_i(q^4; q^4)_j(q^6; q^6)_k} \nonumber \\
	&=\sum_{i, j, k \geq 0}(-1)^k \frac{t^{i+2j+3k}q^{(i+2j+3k)(i+2j+3k-1)+2i+2j+3k^2+6k}}{(q;q)_i(q^4;q^4)_j(q^6;q^6)_k}.
\end{align*}

\noindent The latter generating function is exactly the sum (3.4) which is constructed by Kanade and Russell in \cite{sk-mcr-staircases} (Where they used $x$ for $t$).
\qed

\begin{ex} \normalfont
	Let $\lambda = 2+2+6+12+12+16+18+24+24$. In this case, the almost-seed partition is $\overline{\lambda}= 1+3+6+11+13+16+18+23+25$ and we use the base partition $\beta=1+3+5+7+9+11+13+15+17$. We have $\mu=0+0+1+4+4+5+5+8+8$. Observe that $\mu$ has two non-zero even parts that appear an even number of times, namely $4$ and $8$. Therefore, we can generate $2^2=4$ partitions from the almost-seed partition $\overline{\lambda}$ counted by $kr_2(n, m)$, where $n=116$ and $m=9$. Let us list all of these partitions:
	
	\begin{center}
		$\color{red}{2+2}$ $+$ $6+11+13+16+18+23+25$\\
		$\color{red}{2+2}$ $+$ $6$ $+$ $\color{green}{12+12}$ $+$ $16+18+23+25$\\
		$\color{red}{2+2}$ $+$ $6+11+13+16+18+\color{blue}{24+24}$\\
		$\color{red}{2+2}$ $+$ $6$ $+$ $\color{green}{12+12}$ $+$ $16+18+\color{blue}{24+24}$.
	\end{center}
\end{ex}

\begin{theorem} \cite{sk-mcr-staircases} \label{1, 2, 3 not allowed}
	Consider the partitions satisfying the following conditions:
	
	\begin{enumerate} 
		\item [(a)]  No consecutive parts allowed. 
		\item [(b)] Odd parts do not repeat.
		\item [(c)] For a contiguous sub-partition $\lambda_i+\lambda_{i+1}+\lambda_{i+2}$, we have $| \lambda_i-\lambda_{i+2} | \geq 4 $ if $\lambda_{i+1}$ is even and appears more than once.
		\item [($d''$)]$1, 2$ and $3$ are not allowed to appear as parts.
	\end{enumerate}
	
	\noindent For $n, m \in \mathbb{N}$, let $kr_3(n, m)$ denote the number of partitions of n into m parts such that the partitions satisfy the conditions (a), (b), (c) and (d''). Then,
	\begin{equation} \label{third generating function}
		\sum_{m,n \geq 0}kr_3(n,m)q^nt^m= \sum_{i, j, k \geq 0}(-1)^k \frac{t^{i+2j+3k}q^{(i+2j+3k)(i+2j+3k-1)+4i+6j+3k^2+12k}}{(q;q)_i(q^4;q^4)_j(q^6;q^6)_k}.
	\end{equation} 
\end{theorem}

\begin{proof}
Let $\lambda$ be a partition counted by $kr_2(n, m)$. Then we can find a partition $\lambda'$ counted by $kr_3(n+2m, m)$ iterating each part in $\lambda$ by $2$. Conversely, if $\lambda'$ is a partition counted by $kr_3(n, m)$, we can find a partition $\lambda$ counted by $kr_2(n-2m, m)$ subtracting $2$ from the each part in $\lambda'$. Therefore, there is a $1$-$1$ correspondence between the partitions satisfying the conditions $(a)$, $(b)$, $(c)$, $(d')$ and the partitions satisfying the conditions $(a)$, $(b)$, $(c)$, $(d'')$. Now, by using the generating function in Theorem \ref{1 not allowed}, we can provide a generating function for the partitions satisfying the conditions $(a)$, $(b)$, $(c)$, $(d'')$. For this purpose, we let $t^m \mapsto t^mq^{2m}$ to get that the required generating function is

\begin{align*}
	\sum_{m,n \geq 0}kr_3(n,m)q^nt^m &=\sum_{i, j, k \geq 0}(-1)^k \frac{\textcolor{red}{(tq^2)^{i+2j+3k}} q^{(i+2j+3k)(i+2j+3k-1)+2i+2j+3k^2+6k}}{(q;q)_i(q^4;q^4)_j(q^6;q^6)_k} \nonumber \\
	&=\sum_{i, j, k \geq 0}(-1)^k \frac{t^{i+2j+3k}q^{2i+4j+6k} q^{(i+2j+3k)(i+2j+3k-1)+2i+2j+3k^2+6k}}{(q;q)_i(q^4;q^4)_j(q^6;q^6)_k} \nonumber \\
	&=\sum_{i, j, k \geq 0}(-1)^k \frac{t^{i+2j+3k} q^{(i+2j+3k)(i+2j+3k-1)+4i+6j+3k^2+12k}}{(q;q)_i(q^4;q^4)_j(q^6;q^6)_k}. 
\end{align*}

\noindent The latter generating function is exactly the sum (3.6) which is constructed by Kanade and Russell in \cite{sk-mcr-staircases} (Where they used $x$ for $t$).
\end{proof}

\section{Construction of Evidently Positive Series as Generating Functions}
\label{Evidently Positive Series}

The infinite product

\begin{equation} \label{H0(t;q)}
	H_0(t; q)=\prod_{n = 1}^{\infty} (1+tq^{2n}+ t^2q^{4n})
\end{equation}

which appears as the numerator of $A(t;q;2)$ plays a very important role in the construction of evidently positive series which are presented in this section. 

We try to represent $H_0(t; q)$ as a series with evidently positive coefficients. In this way, we can write an evidently positive series as the generating function of partitions in Theorem \ref{2+2 not allowed}. More precisely, assume that

\begin{equation*} 
	H_0(t; q)= \sum_{n \geq 0} h_nq^nt^n,
\end{equation*}

\noindent where $h_n \geq 0$ for all $n$. Then, we get 

\begin{align*} 
	A(t;q;2) &= \frac{H_0(t; q)}{\prod_{n = 1}^{\infty}(1-tq^{2n-1})(1-t^2q^{4n})}.\frac{1}{(1-t)}\\
	&=\left( \sum_{n \geq 0} h_nq^nt^n\right).\frac{1}{\prod_{n = 1}^{\infty}(1-tq^{2n-1})(1-t^2q^{4n})}.\frac{1}{(1-t)}\\
	&= \left( \sum_{n \geq 0} h_nq^nt^n\right). \frac{1}{(tq;q^2)_\infty(t^2q^4;q^4)_\infty}.\frac{1}{(1-t)}\\
	&= \left( \sum_{n \geq 0} h_nq^nt^n\right).\left(\sum_{i\geq 0} \frac{(tq)^i}{(q^2;q^2)_i} \right).\left( \sum_{j\geq 0}\frac{(t^2q^4)^j}{(q^4;q^4)_j}  \right). \left( \sum_{k\geq 0}t^k \right) \\
	&=\sum_{n,i,j, k \geq 0} \frac{h_nq^{n+i+4j}t^{n+i+2j+k}}{(q^2;q^2)_i(q^4;q^4)_j}.
\end{align*}

\noindent The latter sum is an evidently positive series. After adding the weight of the base partition $\beta$ to this series, we obtain an evidently positive series as the generating function of partitions in Theorem \ref{2+2 not allowed}. Moreover, we can also construct evidently positive series as the generating functions of partitions in Theorem \ref{1 not allowed} and Theorem \ref{1, 2, 3 not allowed} by using the same positive series representation of $H_0(t; q)$. Consider
\begin{equation*} 
	H(t; q)=\prod_{n= 1}^{\infty} (1+tq^{n}+ t^2q^{2n}).
\end{equation*} 

\noindent We observe that $H_0(t; q)=H(t; q^2)$. Therefore, it is sufficient to represent $H(t;q)$ as an evidently positive series. We also observe that $H(t;q)$ is the generating function of partitions with parts occurring at most twice. We now prove Theorem \ref{each part appears at most twice} which gives an evidently positive series as the generating function of partitions  with parts occurring at most twice.

\noindent \emph{\textbf{Proof of Theorem \ref{each part appears at most twice}:}} 
We will show that each partition $\lambda$ counted by $h(n, m)$ corresponds to a triple of partitions $(\beta, \mu, \theta)$ by a series of backward moves, where

\begin{itemize}
	\item $\beta$ is the base partition into $m=2n_2+n_1$ parts, has $n_2$ pairs,  $n_{11}$ immobile singletons, $n_{12}$ moveable singletons, where $n_1=n_{11}+n_{12}$ (pairs, immobile singletons and moveable singletons will be defined shortly), and minimum possible weight,
	\item $\mu$ is a partition into $n_2$ multiples of $3$ ($0$ is allowed as a part),
	\item $\theta$ is a partition into $n_1=n_{11}+n_{12}$ parts ($0$ is allowed as a part), where  $n_{11}$ will be determined once we perform all backward moves on the pairs and the number of $0$'s in $\theta$ $\geq n_{11}$.
	\item $|\lambda|=|\beta|+|\mu|+|\theta|$.
\end{itemize}

\noindent Conversely, we will obtain a unique $\lambda$ from a given triple of partitions $(\beta, \mu, \theta)$ by a series of forward moves. We show that the sequences of backward and forward moves are inverses of each other. Therefore, we will obtain 

\begin{equation*} 
	\sum_{m, n \geq 0}h(n, m)q^nt^m=\sum_{\substack{n_1, n_2 \geq 0 \\ n_1=n_{11}+n_{12} \\ \beta, \mu, \theta}}q^{| \beta | + | \mu |+ | \theta |}t^{2n_2+n_1}
\end{equation*}

\noindent We now define pairs, singletons, the backward moves on pairs and singletons. Let $\lambda$ be a partition counted by $h(n, m)$. The \textit{pairs } in $\lambda$ are defined as a pair of two repeating parts or a pair of two consecutive parts. We pair the leftmost parts first and continue recursively for yet unbound parts. We show the pairs in brackets for convenience. Once we determine all pairs, the remaining parts will be defined as \textit{singletons}. The singletons will be divided into two categories; \textit{immobile singletons} and \textit{moveable singletons}. It is not possible to distinguish a singleton as an immobile or a moveable singleton in the first place. Once we perform all possible backward moves on the pairs (the backward moves on the pairs will be defined shortly), the singletons which exist between the pairs will be immobile singletons and no backward or forward moves will be possible on those singletons. The singletons which appear just after a streak containing pairs and immobile singletons will be moveable singletons. The details will be explained in due course. A pair will not be moved through other pairs in the backward or forward direction, but it can be moved through singletons. We indicate the pair or singleton being moved in boldface, here and elsewhere. The backward moves on pairs are defined as follows:

\vspace{3mm}
\noindent \underline{\textbf{case Ia}}
\begin{center}
	(parts $\leq k-3$), $\mathbf{[k, k+1]}$, (parts $\geq k+1$)
\end{center}
\begin{center}
	\quad \quad \quad \quad \quad \ $\Big\downarrow \text{one backward move}$
\end{center}
\begin{center}
	(parts $\leq k-3$), $\mathbf{[k-1, k-1]}$, (parts $\geq k+1$)
\end{center}

\vspace{3mm}
\noindent \underline{\textbf{case Ib}}
\begin{center}
	(parts $\leq k-4$), $k-2$, $\mathbf{[k, k+1]}$, (parts $\geq k+1$)
\end{center}
\begin{center}
	\quad \quad \quad \quad \quad \ $\Big\downarrow \text{one backward move}$
\end{center}
\begin{center}
	(parts $\leq k-4$), $k-2$, $\mathbf{[k-1, k-1]}$, (parts $\geq k+1$)
\end{center}
\begin{center}
	\quad \quad \quad \quad \quad \ $\Big\downarrow \text{regrouping of pairs}$
\end{center}
\begin{center}
	(parts $\leq k-4$), $\mathbf{[k-2, k-1]}$, $k-1$, (parts $\geq k+1$)
\end{center}

\vspace{3mm}
\noindent \underline{\textbf{case IIa}}

\begin{center}
	(parts $\leq k-4$), $\mathbf{[k, k]}$, (parts $\geq k+1$)
\end{center}
\begin{center}
	\quad \quad \quad \quad \quad \ $\Big\downarrow \text{one backward move}$
\end{center}
\begin{center}
	(parts $\leq k-4$), $\mathbf{[k-2, k-1]}$, (parts $\geq k+1$)
\end{center}

\vspace{3mm}
\noindent \underline{\textbf{case IIb}}
\begin{center}
	(parts $\leq k-4$), $k-2$, $\mathbf{[k, k]}$, (parts $\geq k+1$)
\end{center}
\begin{center}
	\quad \quad \quad \quad \quad \ $\Big\downarrow \text{one backward move}$
\end{center}
\begin{center}
	(parts $\leq k-4$), $k-2$, $\mathbf{[k-2, k-1]}$, (parts $\geq k+1$)
\end{center}
\begin{center}
	\quad \quad \quad \quad \quad \ $\Big\downarrow \text{regrouping of pairs}$
\end{center}
\begin{center}
	(parts $\leq k-4$), $\mathbf{[k-2, k-2]}$, $k-1$, (parts $\geq k+1$)
\end{center}

\vspace{3mm}
\noindent \underline{\textbf{case IIc}}
\begin{center}
	(parts $\leq k-5$), $k-3$, $\mathbf{[k, k]}$, (parts $\geq k+1$)
\end{center}
\begin{center}
	\quad \quad \quad \quad \quad \ $\Big\downarrow \text{one backward move}$
\end{center}
\begin{center}
	(parts $\leq k-5$), $k-3$, $\mathbf{[k-2, k-1]}$, (parts $\geq k+1$)
\end{center}
\begin{center}
	\quad \quad \quad \quad \quad \ $\Big\downarrow \text{regrouping of pairs}$
\end{center}
\begin{center}
	(parts $\leq k-5$), $\mathbf{[k-3, k-2]}$, $k-1$, (parts $\geq k+1$)
\end{center}

\vspace{3mm}
\noindent \underline{\textbf{case IIIa}}

\begin{center}
	$\underbrace{...............................,[k-3, k-2]}_{\substack{\text{a streak of pairs which can not be}\\ \text{moved further back}  }}$, $k-2$, $\mathbf{[k, k+1]}$, $[k+2, k+3]$, (parts $\geq k+3$)
\end{center}
\begin{center}
	\quad \quad \quad \quad \quad  $\Big\downarrow \text{one backward move}$
\end{center}
\begin{center}
	.....................,$[k-3, k-2]$, $k-2$, $\mathbf{[k-1, k-1]}$, $[k+2, k+3]$, (parts $\geq k+3$)
\end{center}

\begin{center}	
	\quad \quad \quad \quad \quad  $\Big\downarrow \text{regrouping the pairs}$
\end{center}
\begin{center}
	.....................,$[k-3, k-2]$, $\underbrace{[k-2, k-1]}_{\substack{\text{can not be moved}\\ \text{further back}}}$, $k-1$, $\mathbf{[k+2, k+3]}$, (parts $\geq k+3$)
\end{center}
\begin{center}
	\quad \quad \quad \quad \quad \quad \quad \quad \quad  $\Big\downarrow \text{one backward move}$
\end{center}
\begin{center}
	.....................,$[k-3, k-2]$, $[k-2, k-1]$, $\underbrace{k-1}_{\substack{\text{an immobile}\\ \text{singleton}}}$, $\mathbf{[k+1, k+1]}$, (parts $\geq k+3$)
\end{center}
The pair  $[k+1, k+1]$ can not be moved further back and the part $k-1$ is an immobile singleton.  If we try to move the pair $[k+1, k+1]$ back once more, we get:

\begin{center}
	.....................,$[k-3, k-2]$, $[k-2, k-1]$, $k-1$, $\mathbf{[k+1, k+1]}$, (parts $\geq k+3$)
\end{center}

\begin{center}
	\quad \quad \quad \quad \quad \quad \quad \quad \quad  $\Big\downarrow \text{one backward move}$
\end{center}

\begin{center}
	.....................,$[k-3, k-2]$, $[k-2, k-1]$, $k-1$, $\mathbf{[k-1, k]}$, (parts $\geq k+3$)
\end{center}
\noindent As we see, the part $k-1$ appears thrice in this case, but appearance of a part more than twice is not allowed.

\vspace{3mm}
\noindent \underline{\textbf{case IIIb}}

\begin{center}
	$\underbrace{...............................,[k-3, k-3]}_{\substack{\text{a streak of pairs which can not be}\\ \text{moved further back}  }}$, $k-2$, $\mathbf{[k, k+1]}$, $[k+2, k+3]$, (parts $\geq k+3$)
\end{center}
\begin{center}
	\quad \quad \quad \quad \quad  $\Big\downarrow \text{one backward move}$
\end{center}
\begin{center}
	.....................,$[k-3, k-3]$, $k-2$, $\mathbf{[k-1, k-1]}$, $[k+2, k+3]$, (parts $\geq k+3$)
\end{center}

\begin{center}	
	\quad \quad \quad \quad \quad  $\Big\downarrow \text{regrouping the pairs}$
\end{center}
\begin{center}
	.....................,$[k-3, k-3]$, $\underbrace{[k-2, k-1]}_{\substack{\text{can not be moved}\\ \text{further back}}}$, $k-1$, $\mathbf{[k+2, k+3]}$, (parts $\geq k+3$)
\end{center}
\begin{center}
	\quad \quad \quad \quad \quad \quad \quad \quad \quad  $\Big\downarrow \text{one backward move}$
\end{center}
\begin{center}
	.....................,$[k-3, k-3]$, $[k-2, k-1]$, $\underbrace{k-1}_{\substack{\text{an immobile}\\ \text{singleton}}}$, $\mathbf{[k+1, k+1]}$, (parts $\geq k+3$)
\end{center}
The pair  $[k+1, k+1]$ can not be moved further back and the part $k-1$ is an immobile singleton as in case \textbf{IIIa}.

\noindent \underline{ \textbf{case IIIc}}

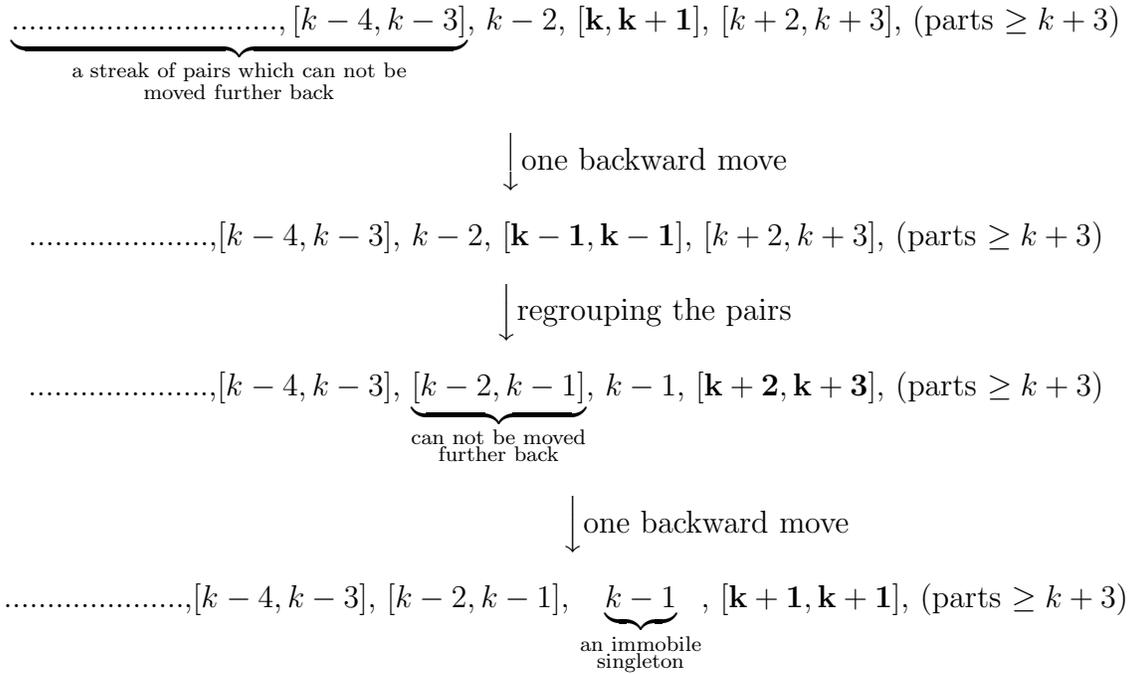
\begin{center}
	$\underbrace{...............................,[k-4, k-3]}_{\substack{\text{a streak of pairs which can not be}\\ \text{moved further back}  }}$, $k-2$, $\mathbf{[k, k+1]}$, $[k+2, k+3]$, (parts $\geq k+3$)
\end{center}
\begin{center}
	\quad \quad \quad \quad \quad  $\Big\downarrow \text{one backward move}$
\end{center}
\begin{center}
	.....................,$[k-4, k-3]$, $k-2$, $\mathbf{[k-1, k-1]}$, $[k+2, k+3]$, (parts $\geq k+3$)
\end{center}

\begin{center}	
	\quad \quad \quad \quad \quad  $\Big\downarrow \text{regrouping the pairs}$
\end{center}
\begin{center}
	.....................,$[k-4, k-3]$, $\underbrace{[k-2, k-1]}_{\substack{\text{can not be moved}\\ \text{further back}}}$, $k-1$, $\mathbf{[k+2, k+3]}$, (parts $\geq k+3$)
\end{center}
\begin{center}
	\quad \quad \quad \quad \quad \quad \quad \quad \quad  $\Big\downarrow \text{one backward move}$
\end{center}
\begin{center}
	.....................,$[k-4, k-3]$, $[k-2, k-1]$, $\underbrace{k-1}_{\substack{\text{an immobile}\\ \text{singleton}}}$, $\mathbf{[k+1, k+1]}$, (parts $\geq k+3$)
\end{center}
The pair  $[k+1, k+1]$ can not be moved further back and the part $k-1$ is an immobile singleton as in the cases \textbf{IIIa} and \textbf{IIIb}.

\vspace{3mm}
\noindent In the cases \textbf{IIIa}, \textbf{IIIb }and \textbf{IIIc}, the streak of pairs at the very beginning, where the pairs are tightly packed, i.e. no backward moves are possible to perform on any of them can not be described more precisely. Because, we have two kinds of pairs; a pair of two repeating parts and a pair of two consecutive parts. Therefore, there may exist three different pairs after a fixed pair. For example, if we consider the pair $[k, k+1]$, there are three possibilities  for the pair just coming after $[k, k+1]$: $[k+1, k+2]$, $[k+2, k+2]$ and $[k+2, k+3]$. 

\vspace{3mm}
\noindent We highlight that an immobile singleton may exist between two pairs such that the first pair is a pair of consecutive parts and the second pair is a pair of repeating parts.

\vspace{3mm}
\noindent We observe that any backward move on a pair decreases the total weight by $3$. Once we determine all pairs and singletons in $\lambda$, we start with the smallest pair and we perform $\frac{\mu_1}{3}$ backward moves on it until it becomes $[1,1]$, $[1,2]$ or $[2,2]$, thus determining $\mu_1$, the smallest part of $\mu$. If the smallest pair is already $[1,1]$, $[1,2]$ or $[2,2]$, then we set $\mu_1=0$. Once the smallest pair is stowed as $[1,1]$, $[1,2]$ or $[2,2]$, we continue with the next smallest pair. We perform $\frac{\mu_2}{3}$ backward moves on the second smallest pair, thus determining $\mu_2$ and so on. When the smallest two pairs are moved back as much as possible, we may have the following partitions at the beginning of the base partition:
\FloatBarrier
\begin{table}[h]
	\begin{center}
		\begin{tabular}{ |c|c|c| } 
			\hline
			$[1,1],[2,2]$& $[1,2],2,[4,4]$ & $[2,2],[3,3]$ \\ 
			$[1,1],[2,3]$ & $[1,2],[2,3]$ & $[2,2],[3,4]$ \\ 
			$[1,1],[3,3]$  & $[1,2],[3,3]$ & $[2,2],[4,4]$ \\ 
			& $[1,2],[3,4]$ &  \\ 
			\hline
		\end{tabular}
	\end{center}
\end{table}
\FloatBarrier

\noindent We now show that any backward move on the smaller pair allows a backward move on the immediately succeeding pair. We may assume that there exist no other parts between the smaller pair and the succeeding pair. Because, if there exist some singletons between two pairs, then we can move the larger pair in the backward direction so that there exist no singletons between the pairs. Here, we also assume that there is no immobile singleton appearing between the pairs during the backward moves.  The case when an immobile singleton exists between two pairs is already considered in the cases \textbf{V} and \textbf{VI}. It is obvious that one backward move on the smaller pair allows one backward move on the larger pair in the cases\textbf{ V }and \textbf{VI}.  We now consider several cases: 

\vspace{3mm}
\noindent \underline{\textbf{case i}}
\begin{center}
	(parts $\leq k-2$), $\mathbf{[k, k+1]}$, $[k+1, k+2]$, (parts $\geq k+2$)
\end{center}
\begin{center}
	\quad \quad \quad \ $\Big\downarrow \text{one backward move}$
\end{center}
\begin{center}
	(parts $\leq k-2$), $[k-1, k-1]$, $\mathbf{[k+1, k+2]}$, (parts $\geq k+2$)
\end{center}
Here, there is a potential regrouping for determining pairs if there is a $k-2$. 
\begin{center}
	\quad \quad \quad \ $\Big\downarrow \text{one backward move}$
\end{center}
\begin{center}
	(parts $\leq k-2$), $[k-1, k-1]$, $\mathbf{[k, k]}$, (parts $\geq k+2$)
\end{center}

\noindent \underline{\textbf{case ii}}
\begin{center}
	(parts $\leq k-2$), $\mathbf{[k, k+1]}$, $[k+2, k+2]$, (parts $\geq k+3$)
\end{center}
\begin{center}
	\quad \quad \quad \ $\Big\downarrow \text{one backward move}$
\end{center}
\begin{center}
	(parts $\leq k-2$), $[k-1, k-1]$, $\mathbf{[k+2, k+2]}$, (parts $\geq k+3$)
\end{center}
Here, there is a potential regrouping for determining pairs if there is a $k-2$. 
\begin{center}
	\quad \quad \quad \ $\Big\downarrow \text{one backward move}$
\end{center}
\begin{center}
	(parts $\leq k-2$), $[k-1, k-1]$, $\mathbf{[k, k+1]}$, (parts $\geq k+3$)
\end{center}

\newpage
\noindent \underline{\textbf{case iii}}
\begin{center}
	(parts $\leq k-2$), $\mathbf{[k, k+1]}$, $[k+2, k+3]$, (parts $\geq k+3$)
\end{center}
\begin{center}
	\quad \quad \quad \ $\Big\downarrow \text{one backward move}$
\end{center}
\begin{center}
	(parts $\leq k-2$), $[k-1, k-1]$, $\mathbf{[k+2, k+3]}$, (parts $\geq k+3$)
\end{center}
Here, there is a potential regrouping for determining pairs if there is a $k-2$. 
\begin{center}
	\quad \quad \quad \ $\Big\downarrow \text{one backward move}$
\end{center}
\begin{center}
	(parts $\leq k-2$), $[k-1, k-1]$, $\mathbf{[k+1, k+1]}$, (parts $\geq k+3$)
\end{center}

\vspace{3mm}
\noindent \underline{\textbf{case iv}}
\begin{center}
	(parts $\leq k-2$), $\mathbf{[k, k]}$, $[k+1, k+1]$, (parts $\geq k+2$)
\end{center}
\begin{center}
	\quad \quad \quad \ $\Big\downarrow \text{one backward move}$
\end{center}
\begin{center}
	(parts $\leq k-2$), $[k-2, k-1]$, $\mathbf{[k+1, k+1]}$, (parts $\geq k+2$)
\end{center}
Here, there is a potential regrouping for determining pairs if there is a $k-2$ or $k-3$. 
\begin{center}
	\quad \quad \quad \ $\Big\downarrow \text{one backward move}$
\end{center}
\begin{center}
	(parts $\leq k-2$), $[k-2, k-1]$, $\mathbf{[k-1, k]}$, (parts $\geq k+2$)
\end{center}

\vspace{3mm}
\noindent \underline{\textbf{case v}}
\begin{center}
	(parts $\leq k-2$), $\mathbf{[k, k]}$, $[k+1, k+2]$, (parts $\geq k+2$)
\end{center}
\begin{center}
	\quad \quad \quad \ $\Big\downarrow \text{one backward move}$
\end{center}
\begin{center}
	(parts $\leq k-2$), $[k-2, k-1]$, $\mathbf{[k+1, k+2]}$, (parts $\geq k+2$)
\end{center}
Here, there is a potential regrouping for determining pairs if there is a $k-2$ or $k-3$. 
\begin{center}
	\quad \quad \quad \ $\Big\downarrow \text{one backward move}$
\end{center}
\begin{center}
	(parts $\leq k-2$), $[k-1, k-1]$, $\mathbf{[k, k]}$, (parts $\geq k+2$)
\end{center}

\vspace{3mm}
\noindent \underline{\textbf{case vi}}
\begin{center}
	(parts $\leq k-2$), $\mathbf{[k, k]}$, $[k+2, k+2]$, (parts $\geq k+3$)
\end{center}
\begin{center}
	\quad \quad \quad \ $\Big\downarrow \text{one backward move}$
\end{center}
\begin{center}
	(parts $\leq k-2$), $[k-2, k-1]$, $\mathbf{[k+2, k+2]}$, (parts $\geq k+3$)
\end{center}
Here, there is a potential regrouping for determining pairs if there is a $k-2$ or $k-3$. 
\begin{center}
	\quad \quad \quad \ $\Big\downarrow \text{one backward move}$
\end{center}
\begin{center}
	(parts $\leq k-2$), $[k-1, k-1]$, $\mathbf{[k, k+1]}$, (parts $\geq k+3$)
\end{center}
In all cases \textbf{i}, \textbf{ii}, \textbf{iii}, \textbf{iv}, \textbf{v} and \textbf{vi}, one backward move on the smaller pair allows a backward move on the larger pair. Thus, it is obvious that
\begin{center}
	$\mu_1 \leq \mu_2 \leq \ldots \leq \mu_{n_2}$.
\end{center}

\noindent Once we perform all possible backward moves on the pairs, we will  have pairs which are fixed in their places, i.e. they can not be moved further back and some immobile singletons may exist between the pairs.  More precisely, once we perform all possible backward moves on the pairs, the intermediate partition looks like
\begin{center}
	$\underbrace{...................................................}_{\substack{\text{$n_2-1$ pairs which can not be moved further back}\\ \text{and $n_{11}$ immobile singletons}}}$, $[k, k]$, moveable singletons $\geq k+1$
\end{center}
\begin{center}
	or
\end{center}
\begin{center}
	$\underbrace{...................................................}_{\substack{\text{$n_2-1$ pairs which can not be moved further back}\\ \text{and $n_{11}$ immobile singletons}}}$, $[k, k+1]$, moveable singletons $\geq k+1$
\end{center}

\noindent To make the $n_{12}$ moveable singletons $k+1, k+3, k+5, \ldots, k+2n_{12}-1$, we move the singletons one by one in backward direction. We subtract $\theta_{n_{11}+1}$ from the smallest moveable singleton, $\theta_{n_{11}+2}$ from the next smallest moveable singleton and so on. Since we have $n_{11}$ immobile singletons, the backward moves of those singletons will correspond to $n_{11}$ $0$'s in the partition $\theta$. It is clear that any backward move on a moveable singleton allows a backward move on the immediately succeeding moveable singleton. It is immediate that 
\begin{center}
	$\theta_1 =\theta_2= \ldots=\theta_{n_{11}}=0 \leq \theta_{n_{11}+1} \leq \theta_{n_{11}+2} \leq  \ldots \leq \theta_{n_{11}+n_{12}}=\theta_{n_1}$.
\end{center}

\noindent We have now produced a triple of partitions $(\beta, \mu, \theta)$ from the given partition $\lambda$ counted by $h(n,m)$.

\vspace{3mm}
\noindent Conversely, given a triple of partitions $(\beta, \mu, \theta)$, we will produce a unique partition $\lambda$ counted by $h(n,m)$ via forward moves. 

\vspace{3mm}
\noindent Let $(\beta, \mu, \theta)$ be a triple of partitions such that  $\beta$ is the base partition into $m=2n_2+n_1$ parts, has $n_2$ pairs, $n_{11}$ immobile singletons, $n_{12}$ moveable singletons, where $n_1=n_{11}+n_{12}$ and minimum possible weight, $\mu$ is a partition into $n_2$ multiples of $3$ ($0$ is allowed as a part), $\theta$ is a partition into $n_1=n_{11}+n_{12}$ parts ($0$ is allowed as a part), where the number of $0$'s in $\theta$ $\geq n_{11}$. We first add the $i$th largest part of $\theta$ to the $i$th largest singleton in $\beta$ for $i=1, 2, \ldots, n_{12}$ in this order. These are the forward moves on the moveable singletons.  Once we perform all forward moves on the moveable singletons, we perform $\frac{1}{3}$.(the $i$th largest part of $\mu$) forward moves on the $i$th largest pair in $\beta$ for $i=1, 2, \ldots, n_2$, in this order. As we showed for the backward moves on the pairs , one can easily show that one forward move on the larger pair allows one forward move of the preceding smaller pair. Any forward move on a pair will increase the total weight by $3$. The forward moves on the pairs are defined as follows:

\vspace{3mm}
\noindent \underline{\textbf{case I'a}}
\begin{center}
	(parts $\leq k-3$), $\mathbf{[k-1, k-1]}$, (parts $\geq k+1$)
\end{center}
\begin{center}
	\quad \quad \quad \quad \quad \ $\Big\downarrow \text{one forward move}$
\end{center}
\begin{center}
	(parts $\leq k-3$), $\mathbf{[k, k+1]}$, (parts $\geq k+1$)
\end{center}
\noindent Here, there is a potential regrouping for determining pairs if there is a $k+1$ or $k+2$.

\vspace{3mm}
\noindent \underline{\textbf{case I'b}}

\begin{center}
	(parts $\leq k-4$), $\mathbf{[k-2, k-1]}$, $k-1$, (parts $\geq k+1$)
\end{center}
\begin{center}
	\quad \quad \quad \quad \quad \ $\Big\downarrow \text{regrouping of pairs}$
\end{center}
\begin{center}
	(parts $\leq k-4$), $k-2$, $\mathbf{[k-1, k-1]}$,  (parts $\geq k+1$)
\end{center}
\begin{center}
	\quad \quad \quad \quad \quad \ $\Big\downarrow \text{one forward move}$
\end{center}
\begin{center}
	(parts $\leq k-4$), $k-2$, $\mathbf{[k, k+1]}$,  (parts $\geq k+1$)
\end{center}
\noindent Here, there is a potential regrouping for determining pairs if there is a $k+1$ or $k+2$.

\vspace{3mm}
\noindent \underline{\textbf{case II'a}}
\begin{center}
	(parts $\leq k-4$), $\mathbf{[k-2, k-1]}$, (parts $\geq k+1$)
\end{center}
\begin{center}
	\quad \quad \quad \quad \quad \ $\Big\downarrow \text{one forward move}$
\end{center}
\begin{center}
	(parts $\leq k-4$), $\mathbf{[k, k]}$, (parts $\geq k+1$)
\end{center}
\noindent Here again, there is a potential regrouping for determining pairs if there is a $k+1$.

\vspace{3mm}
\noindent \underline{\textbf{case II'b}}

\begin{center}
	(parts $\leq k-4$), $\mathbf{[k-2, k-2]}$, $k-1$, (parts $\geq k+1$)
\end{center}
\begin{center}
	\quad \quad \quad \quad \quad \ $\Big\downarrow \text{regrouping of pairs}$
\end{center}
\begin{center}
	(parts $\leq k-4$), $k-2$, $\mathbf{[k-2, k-1]}$, (parts $\geq k+1$)
\end{center}
\begin{center}
	\quad \quad \quad \quad \quad \ $\Big\downarrow \text{one forward move}$
\end{center}
\begin{center}
	(parts $\leq k-4$), $k-2$, $\mathbf{[k, k]}$, (parts $\geq k+1$)
\end{center}

\noindent Here again, there is a potential regrouping for determining pairs if there is a $k+1$.

\vspace{3mm}
\noindent \underline{\textbf{case II'c}}

\begin{center}
	(parts $\leq k-5$), $\mathbf{[k-3, k-2]}$, $k-1$, (parts $\geq k+1$)
\end{center}
\begin{center}
	\quad \quad \quad \quad \quad \ $\Big\downarrow \text{regrouping of pairs}$
\end{center}
\begin{center}
	(parts $\leq k-5$), $k-3$, $\mathbf{[k-2, k-1]}$, (parts $\geq k+1$)
\end{center}
\begin{center}
	\quad \quad \quad \quad \quad \ $\Big\downarrow \text{one forward move}$
\end{center}

\begin{center}
	(parts $\leq k-5$), $k-3$, $\mathbf{[k, k]}$, (parts $\geq k+1$)
\end{center}
\noindent Here again, there is a potential regrouping for determining pairs if there is a $k+1$.

\vspace{3mm}
\noindent \underline{\textbf{case III'a}}

\begin{center}
	$\underbrace{.....................,[k-3, k-2], [k-2, k-1]}_{\substack{\text{a streak of pairs which are}\\ \text{tightly packed}}}$, $k-1$, $\mathbf{[k+1, k+1]}$, (parts $\geq k+3$)
\end{center}
\begin{center}
	\quad \quad \quad \quad \quad \quad \quad \quad \quad  $\Big\downarrow \text{one forward move}$
\end{center}
\begin{center}
	$.....................,[k-3, k-2], [k-2, k-1]$, $k-1$, $\mathbf{[k+2, k+3]}$, (parts $\geq k+3$)
\end{center}

\noindent Here, there is a potential regrouping of the pairs if there is a $k+3$ or $k+4$. We now perform a forward move on the preceding smaller  pair. We first regroup the pairs and we perform a forward move on the new pair:

\begin{center}
	$.....................,[k-3, k-2]$, $[k-2,$ $\underbrace{k-1],k-1}$, $[k+2, k+3]$, (parts $\geq k+3$)
\end{center}

\begin{center}
	\quad \quad \quad \quad  $\Big\downarrow \text{regrouping the pairs}$
\end{center}
\begin{center}
	$.....................,[k-3, k-2]$, $k-2$,  $\mathbf{[k-1,k-1]}$, $[k+2, k+3]$, (parts $\geq k+3$)
\end{center}
\begin{center}
	\quad \quad \quad \quad  $\Big\downarrow \text{one forward move}$
\end{center}
\begin{center}
	$.....................,[k-3, k-2]$, $k-2$,  $\mathbf{[k,k+1]}$, $[k+2, k+3]$, (parts $\geq k+3$)
\end{center}

\noindent \underline{\textbf{case III'b}}

\begin{center}
	$\underbrace{.....................,[k-3, k-3], [k-2, k-1]}_{\substack{\text{a streak of pairs which are}\\ \text{tightly packed}}}$, $k-1$, $\mathbf{[k+1, k+1]}$, (parts $\geq k+3$)
\end{center}
\begin{center}
	\quad \quad \quad \quad \quad \quad \quad \quad \quad  $\Big\downarrow \text{one forward move}$
\end{center}
\begin{center}
	$.....................,[k-3, k-3], [k-2, k-1]$, $k-1$, $\mathbf{[k+2, k+3]}$, (parts $\geq k+3$)
\end{center}

\noindent Here, there is a potential regrouping of the pairs if there is a $k+3$ or $k+4$. We now perform a forward move on the preceding smaller  pair. We first regroup the pairs and we perform a forward move on the new pair:

\begin{center}
	$.....................,[k-3, k-3]$, $[k-2,$ $\underbrace{k-1],k-1}$, $[k+2, k+3]$, (parts $\geq k+3$)
\end{center}

\begin{center}
	\quad \quad \quad \quad  $\Big\downarrow \text{regrouping the pairs}$
\end{center}
\begin{center}
	$.....................,[k-3, k-3]$, $k-2$,  $\mathbf{[k-1,k-1]}$, $[k+2, k+3]$, (parts $\geq k+3$)
\end{center}
\begin{center}
	\quad \quad \quad \quad  $\Big\downarrow \text{one forward move}$
\end{center}
\begin{center}
	$.....................,[k-3, k-3]$, $k-2$,  $\mathbf{[k,k+1]}$, $[k+2, k+3]$, (parts $\geq k+3$)
\end{center}

\noindent \underline{\textbf{case III'c}}

\begin{center}
	$\underbrace{.....................,[k-4, k-3], [k-2, k-1]}_{\substack{\text{a streak of pairs which are}\\ \text{tightly packed}}}$, $k-1$, $\mathbf{[k+1, k+1]}$, (parts $\geq k+3$)
\end{center}
\begin{center}
	\quad \quad \quad \quad \quad \quad \quad \quad \quad  $\Big\downarrow \text{one forward move}$
\end{center}
\begin{center}
	$.....................,[k-4, k-3], [k-2, k-1]$, $k-1$, $\mathbf{[k+2, k+3]}$, (parts $\geq k+3$)
\end{center}

\noindent Here, there is a potential regrouping of the pairs if there is a $k+3$ or $k+4$. We now perform a forward move on the preceding smaller  pair. We first regroup the pairs and we perform a forward move on the new pair:

\begin{center}
	$.....................,[k-4, k-3]$, $[k-2,$ $\underbrace{k-1],k-1}$, $[k+2, k+3]$, (parts $\geq k+3$)
\end{center}

\begin{center}
	\quad \quad \quad \quad  $\Big\downarrow \text{regrouping the pairs}$
\end{center}
\begin{center}
	$.....................,[k-4, k-3]$, $k-2$,  $\mathbf{[k-1,k-1]}$, $[k+2, k+3]$, (parts $\geq k+3$)
\end{center}
\begin{center}
	\quad \quad \quad \quad  $\Big\downarrow \text{one forward move}$
\end{center}
\begin{center}
	$.....................,[k-4, k-3]$, $k-2$,  $\mathbf{[k,k+1]}$, $[k+2, k+3]$, (parts $\geq k+3$)
\end{center}

\noindent We notice that the corresponding cases for the backward and forward moves have switched inputs and outputs and we perform the backward moves and the forward moves in the exact reverse order. Therefore, $\lambda$'s enumerated by $h(n, m)$ are in $1$-$1$ correspondence with the triples $(\beta, \mu, \theta)$. 

\vspace{3mm}
\noindent It follows that
\begin{equation*} 
	\sum_{m, n \geq 0}h(n, m)q^nt^m=\sum_{\substack{n_1, n_2 \geq 0 \\ n_1=n_{11}+n_{12} \\ \beta, \mu, \theta}}q^{| \beta | +| \mu |+| \theta |}t^{2n_2+n_1}=\sum_{\substack{n_1, n_2 \geq 0 \\ n_1=n_{11}+n_{12} \\ \beta}}\frac{q^{| \beta |}t^{2n_2+n_1}}{(q; q)_{n_{12}}(q^3; q^3)_{n_2}}.
\end{equation*}
\qed

\begin{ex} \normalfont
	Let $\lambda=1, 4, 4, 5, 6, 6, 9, 10 ,11, 12, 12, 14$. We first determine the pairs and the singletons:
	\begin{center}
		$\lambda=1, [4,4], [5,6], 6, [9, 10], [11, 12], 12, 14$.
	\end{center}
	
	\noindent We have $n_2=4$ pairs and $n_1=4$ singletons. We start to perform backward moves on the pairs, and once all possible backward moves are performed on the pairs, we continue with the backward moves on the moveable singletons.
	\begin{center} 
		$\lambda=1, \mathbf{[4, 4]}, [5, 6], 6, [9, 10], [11, 12], 12, 14$
	\end{center}
	\begin{center}
		$\Big\downarrow \text{one backward move}$
	\end{center}
	\begin{center} 
		$1, \mathbf{[2, 3]}, [5, 6], 6, [9, 10], [11, 12], 12, 14$
	\end{center}
	\begin{center}
		\quad \quad \quad	$\Big\downarrow \text{regrouping the pairs}$
	\end{center}
	\begin{center} 
		$[1, 2], 3, \mathbf{[5, 6]}, 6, [9, 10], [11, 12], 12, 14$
	\end{center}
	\begin{center}
		$\Big\downarrow \text{one backward move}$
	\end{center}
	\begin{center} 
		$[1, 2], 3, \mathbf{[4, 4]}, 6, [9, 10], [11, 12], 12, 14$
	\end{center}
	\begin{center}
		\quad \quad \quad	$\Big\downarrow \text{regrouping the pairs}$
	\end{center}
	\begin{center} 
		$[1, 2], [3, 4], 4, 6, \mathbf{[9, 10]}, [11, 12], 12, 14$
	\end{center}
	\begin{center}
		$\Big\downarrow \text{one backward move}$
	\end{center}
	\begin{center} 
		$[1, 2], [3, 4], 4, 6, \mathbf{[8, 8]}, [11, 12], 12, 14$
	\end{center}
	\begin{center}
		$\Big\downarrow \text{one backward move}$
	\end{center}
	\begin{center} 
		$[1, 2], [3, 4], 4, 6, \mathbf{[6, 7]}, [11, 12], 12, 14$
	\end{center}
	\begin{center}
		\quad \quad \quad	$\Big\downarrow \text{regrouping the pairs}$
	\end{center}
	\begin{center} 
		$[1, 2], [3, 4], 4, [6, 6], 7, \mathbf{[11, 12]}, 12, 14$
	\end{center}
	\noindent We observe that the pair $[6, 6]$ can not be moved further back and $4$ is an immobile singleton. 
	\begin{center}
		$\Big\downarrow \text{one backward move}$
	\end{center}
	\begin{center} 
		$[1, 2], [3, 4], 4, [6, 6], 7, \mathbf{[10, 10]}, 12, 14$
	\end{center}
	\begin{center}
		$\Big\downarrow \text{one backward move}$
	\end{center}
	\begin{center} 
		$[1, 2], [3, 4], 4, [6, 6], 7, \mathbf{[8, 9]}, 12, 14$
	\end{center}
	\begin{center}
		\quad \quad \quad	$\Big\downarrow \text{regrouping the pairs}$
	\end{center}
	\begin{center} 
		$[1, 2], [3, 4], 4, [6, 6], [7, 8], \mathbf{9}, 12, 14$
	\end{center}
	\begin{center}
		$\Big\downarrow \text{one backward move}$
	\end{center}
	\begin{center} 
		$[1, 2], [3, 4], 4, [6, 6], [7, 8], 8, \mathbf{12}, 14$
	\end{center}
	\begin{center}
		$\Big\downarrow \text{two backward moves}$
	\end{center}
	\begin{center} 
		$[1, 2], [3, 4], 4, [6, 6], [7, 8], 8, 10, \mathbf{14}$
	\end{center}
	\begin{center}
		$\Big\downarrow \text{two backward moves}$
	\end{center}
	\begin{center} 
		$\beta=[1, 2], [3, 4], 4, [6, 6], [7, 8], 8, 10, 12$
	\end{center}
	
	\noindent We have $\mu=3+3+6+6$ and $\theta=0+1+2+2$. We observe that
	\begin{center}
		$|\lambda|=94=|\beta|+|\mu|+|\theta|=71+18+5$.
	\end{center}
\end{ex}

\begin{ex} \normalfont
	Let $\beta= [2, 2], [3, 4], 4, [6, 6], [7, 8], 8, [10, 10], 11, 13, 15$ be a base partition with 5 pairs, 2 immobile singletons and 3 moveable singletons, $\mu=3+3+3+6+6$ and $\theta=0+0+2+3+5$. 
	After incorporating parts of $\theta$ as forward moves on the moveable singletons, the intermediate partition becomes
	\begin{center}
		$[2, 2], [3, 4], 4, [6, 6], [7, 8], 8, [10, 10], 13, 16, 20$.
	\end{center}
	We perform $\frac{\mu_5}{3}=2$ forward moves on the largest pair $[10, 10]$:
	\begin{center}
		$[2, 2], [3, 4], 4, [6, 6], [7, 8], 8, \mathbf{[10, 10]}, 13, 16, 20$
	\end{center}
	\begin{center}
		$\Big\downarrow \text{one forward move}$
	\end{center}
	\begin{center}
		$[2, 2], [3, 4], 4, [6, 6], [7, 8], 8, \mathbf{[11, 12]}, 13, 16, 20$
	\end{center}
	\begin{center}
		$\Big\downarrow \text{regrouping the pairs}$
	\end{center}
	\begin{center}
		$[2, 2], [3, 4], 4, [6, 6], [7, 8], 8, 11, \mathbf{[12, 13]}, 16, 20$
	\end{center}
	\begin{center}
		$\Big\downarrow \text{one forward move}$
	\end{center}
	\begin{center}
		$[2, 2], [3, 4], 4, [6, 6], \mathbf{[7, 8]}, 8, 11, [14, 14], 16, 20$
	\end{center}
	We now perform $\frac{\mu_4}{3}=2$ forward moves on the second largest pair. We first need to regroup the pairs. 
	\begin{center}
		$\Big\downarrow \text{regrouping of pairs}$
	\end{center}
	\begin{center}
		$[2, 2], [3, 4], 4, [6, 6], 7, \mathbf{[8, 8]}, 11, [14, 14], 16, 20$
	\end{center}
	\begin{center}
		$\Big\downarrow \text{one forward move}$
	\end{center}
	\begin{center}
		$[2, 2], [3, 4], 4, [6, 6], 7, \mathbf{[9, 10]}, 11, [14, 14], 16, 20$
	\end{center}
	\begin{center}
		$\Big\downarrow \text{regrouping of pairs}$
	\end{center}
	\begin{center}
		$[2, 2], [3, 4], 4, [6, 6], 7, 9, \mathbf{[10, 11]}, [14, 14], 16, 20$
	\end{center}
	\begin{center}
		$\Big\downarrow \text{one forward move}$
	\end{center}
	\begin{center}
		$[2, 2], [3, 4], 4, \mathbf{[6, 6]}, 7, 9, [12, 12], [14, 14], 16, 20$
	\end{center}
	We now perform $\frac{\mu_3}{3}=1$ forward move on the third largest pair. We first need to regroup the pairs. 
	\begin{center}
		$\Big\downarrow \text{regrouping of pairs}$
	\end{center}
	\begin{center}
		$[2, 2], [3, 4], 4, 6, \mathbf{[6, 7]}, 9, [12, 12], [14, 14], 16, 20$
	\end{center}
	\begin{center}
		$\Big\downarrow \text{one forward move}$
	\end{center}
	\begin{center}
		$[2, 2], \mathbf{[3, 4]}, 4, 6, [8, 8], 9, [12, 12], [14, 14], 16, 20$
	\end{center}
	We now perform $\frac{\mu_2}{3}=1$ forward move on the fourth largest pair. We first need to regroup the pairs. 
	\begin{center}
		$\Big\downarrow \text{regrouping of pairs}$
	\end{center}
	\begin{center}
		$[2, 2], 3, \mathbf{[4, 4]}, 6, [8, 8], 9, [12, 12], [14, 14], 16, 20$
	\end{center}
	\begin{center}
		$\Big\downarrow \text{one forward move}$
	\end{center}
	\begin{center}
		$\mathbf{[2, 2]}, 3, [5, 6], 6, [8, 8], 9, [12, 12], [14, 14], 16, 20$
	\end{center}
	We now perform $\frac{\mu_1}{3}=1$ forward move on the smallest pair. We first need to regroup the pairs. 
	\begin{center}
		$\Big\downarrow \text{regrouping of pairs}$
	\end{center}
	\begin{center}
		$2, \mathbf{[2, 3]}, [5, 6], 6, [8, 8], 9, [12, 12], [14, 14], 16, 20$
	\end{center}
	\begin{center}
		$\Big\downarrow \text{one forward move}$
	\end{center}
	\begin{center}
		$2, \mathbf{[4, 4]}, [5, 6], 6, [8, 8], 9, [12, 12], [14, 14], 16, 20$
	\end{center}
	This final partition is $\lambda$. Its weight is $140=|\lambda|=|\beta|+|\mu|+|\theta|=109+21+10$ indeed. 
\end{ex}

\noindent The only shortcoming in the  generating function (\ref{each part appears at most twice-generating function}) is that $|\beta|$, the weight of the base partition $\beta$,  is not given explicitly. To give an explicit formula for $|\beta|$ with respect to the number of pairs and the number of singletons has not been possible. Because, for a fixed number of pairs and singletons, there exist different base partitions. But, we will be able to construct some polynomials as generating functions of the base partitions with a given number of pairs and singletons. For this purpose, we prove Lemma \ref{polynomials}.

\noindent \emph{\textbf{Proof of Lemma \ref{polynomials}:}} We first prove the equation (\ref{P=P0+P1}). Let $\beta$ be a base partition counted by $P(m_1,m_2, m_3, m+1; q)$. Then,  it has one of the following forms:
\begin{equation*} \label{[m,m]}
	\underbrace{..........................................}_{\substack{\text{$m_1-1$ pairs of two repeating parts}\\ \text{$m_2$ pairs of two consecutive parts} \\ \text{$m_3$ blocks}}}, [m,m], \underbrace{\color{blue}m+1, m+3, \ldots, m+2n_{12}-1}_{\substack{\text{$n_{12}$ moveable singletons}}}   \tag{$\ast$}
\end{equation*}
\begin{center}
	or
\end{center}
\begin{equation*}\label{[m,m+1]}
	\underbrace{..........................................}_{\substack{\text{$m_1$ pairs of two repeating parts}\\ \text{$m_2-1$ pairs of two consecutive parts} \\ \text{$m_3$ blocks}}}, [m,m+1], \underbrace{\color{blue}m+1, m+3, \ldots,m+2n_{12}-1}_{\substack{\text{$n_{12}$ moveable singletons}}}    \tag{$\ast\ast$}
\end{equation*}

\noindent The partition in (\ref{[m,m]}) is a base partition counted by $P_0(m_1,m_2, m_3, m+1; q)$ and the partition in (\ref{[m,m+1]}) is a base partition counted by $P_1(m_1,m_2, m_3, m+1; q)$. It is clear that there is no base partition which is counted by both $P_0(m_1,m_2, m_3, m+1; q)$ and $P_1(m_1,m_2, m_3, m+1; q)$. Hence, (\ref{P=P0+P1}) follows.

\vspace{3mm}
\noindent We now prove (\ref{functional equation of P0}). Let $\beta$ be a base partition counted by $P_0(m_1,m_2, m_3,m+1; q)$. $\beta$ has the form in (\ref{[m,m]}). The largest pair $[m, m]$ may be contained in a block or not. We consider both cases separately. If the largest pair $[m, m]$ is not contained in a block and we delete it from $\beta$, the remaining partition may have one of the following forms:
\begin{equation*} \label{case a}
	\underbrace{..........................................}_{\substack{\text{$m_1-2$ pairs of two repeating parts}\\ \text{$m_2$ pairs of two consecutive parts} \\ \text{$m_3$ blocks}}}, [m-1,m-1], \underbrace{\color{blue}m, m+2, \ldots, m+2n_{12}-2}_{\substack{\text{$n_{12}$ moveable singletons}}}    \tag{a}
\end{equation*}
\begin{equation*} \label{case b}
	\underbrace{..........................................}_{\substack{\text{$m_1-1$ pairs of two repeating parts}\\ \text{$m_2-1$ pairs of two consecutive parts} \\ \text{$m_3$ blocks}}}, [m-2,m-1], \underbrace{\color{blue}m-1, m+1, \ldots, m+2n_{12}-3}_{\substack{\text{$n_{12}$ moveable singletons}}}    \tag{b}
\end{equation*}
\begin{equation*}\label{case c}
	\underbrace{..........................................}_{\substack{\text{$m_1-2$ pairs of two repeating parts}\\ \text{$m_2$ pairs of two consecutive parts} \\ \text{$m_3$ blocks}}}, [m-2,m-2], \underbrace{\color{blue}m-1, m+1, \ldots, m+2n_{12}-3}_{\substack{\text{$n_{12}$ moveable singletons}}}    \tag{c}
\end{equation*}

\noindent If the largest pair $[m, m]$ is contained in a block, then we have:
\begin{equation*} \label{[m,m] contained in a block}
	\underbrace{................................}_{\substack{\text{$m_1$ pairs of }\\ \text{two repeating parts}\\ \text{$m_2$ pairs of } \\ \text{two consecutive parts}\\\text{$m_3-1$ blocks}}}, \underbrace{[m-3, m-2], m-2, [m,m]}_{\substack{\text{a block}}}, \underbrace{\color{blue}m+1,  \ldots,m+2n_{12}-1}_{\substack{\text{$n_{12}$ moveable singletons}}} \tag{$\ast\ast\ast$}
\end{equation*}

\noindent If we delete the block $[m-3, m-2], m-2, [m,m]$ in (\ref{[m,m] contained in a block}), the remaining partition may have one of the following forms:
\begin{equation*} \label{case a'}
	\underbrace{..........................................}_{\substack{\text{$m_1$ pairs of two repeating parts}\\ \text{$m_2-1$ pairs of two consecutive parts} \\ \text{$m_3-1$ blocks}}}, [m-4,m-3], \underbrace{\color{blue}m-3, m-1, \ldots, m+2n_{12}-5}_{\substack{\text{$n_{12}$ moveable singletons}}}    \tag{a'}
\end{equation*}

\begin{equation*} \label{case b'}
	\underbrace{..........................................}_{\substack{\text{$m_1-1$ pairs of two repeating parts}\\ \text{$m_2$ pairs of two consecutive parts} \\ \text{$m_3-1$ blocks}}}, [m-4,m-4], \underbrace{\color{blue}m-3, m-1,  m+2n_{12}-5}_{\substack{\text{$n_{12}$ moveable singletons}}}  \tag{b'}
\end{equation*}

\begin{equation*} \label{case c'}
	\underbrace{..........................................}_{\substack{\text{$m_1$ pairs of two repeating parts}\\ \text{$m_2-1$ pairs of two consecutive parts} \\ \text{$m_3-1$ blocks}}}, [m-5,m-4], \underbrace{\color{blue}m-4, m-2, \ldots, m+2n_{12}-6}_{\substack{\text{$n_{12}$ moveable singletons}}}    \tag{c'}
\end{equation*}

\noindent The base partitions in (\ref{case a}), (\ref{case b}), (\ref{case c}), (\ref{case a'}), (\ref{case b'}) and (\ref{case c'}) are  counted by $P_0(m_1-1, m_2, m_3, \\ m; q)$, $P_1(m_1-1, m_2, m_3, m-1; q)$, $P_0(m_1-1, m_2, m_3, m-1; q)$, $P_1(m_1, m_2, m_3-1, m-3; q)$, $P_0(m_1, m_2, m_3-1,m-3; q)$ and $P_1(m_1, m_2, m_3-1, m-4; q)$, respectively. Therefore, we have:
\begin{align*} 
	P_0(m_1,m_2, m_3, m+1; q)= &\underbrace{q^{2m}}_{\substack{\text{for the deleted } \\ \text{pair $[m, m]$}}}\Big[P_0(m_{1}-1,m_2, m_3, m; q) \nonumber \\&+P_1(m_{1}-1,m_2, m_3, m-1; q) \nonumber
	\\&+P_0(m_{1}-1,m_2, m_3, m-1; q)\Big] \nonumber
	\\&+\underbrace{q^{5m-7}}_{\substack{\text{for the deleted block} \\ \text{ $[m-3, m-2], m-2, [m, m] $}}}\Big[P_1(m_1,m_2, m_3-1, m-3; q) \nonumber 
	\\&+P_0(m_1,m_2, m_3-1, m-3; q) \nonumber
	\\&+P_1(m_1,m_2, m_3-1, m-4; q)\Big]
\end{align*}

\noindent We now prove (\ref{functional equation of P1}). Let $\beta$ be a base partition counted by $P_1(m_1,m_2, m_3, m+1; q)$. $\beta$ has the form in (\ref{[m,m+1]}). If we delete the largest pair $[m,m+1]$, the remaining partition  may have one of the following forms:
\begin{equation*} \label{case d}
	\underbrace{..........................................}_{\substack{\text{$m_1$ pairs of two repeating parts}\\ \text{$m_2-2$ pairs of two consecutive parts} \\ \text{$m_3$ blocks}}}, [m-1,m], \underbrace{\color{blue}m, m+2, \ldots, m+2n_{12}-2}_{\substack{\text{$n_{12}$ moveable singletons}}}     \tag{d}
\end{equation*}
\begin{equation*} \label{case e}
	\underbrace{..........................................}_{\substack{\text{$m_1-1$ pairs of two repeating parts}\\ \text{$m_2-1$ pairs of two consecutive parts} \\ \text{$m_3$ blocks}}}, [m-1,m-1], \underbrace{\color{blue}m, m+2, \ldots, m+2n_{12}-2}_{\substack{\text{$n_{12}$ moveable singletons}}}    \tag{e}
\end{equation*}
\begin{equation*} \label{case f}
	\underbrace{..........................................}_{\substack{\text{$m_1$ pairs of two repeating parts}\\ \text{$m_2-2$ pairs of two consecutive parts} \\ \text{$m_3$ blocks}}}, [m-2,m-1], \underbrace{\color{blue}m-1, m+1, \ldots, m+2n_{12}-3}_{\substack{\text{$n_{12}$ moveable singletons}}}   \tag{f}
\end{equation*}

\noindent The base partitions in (\ref{case d}), (\ref{case e}) and (\ref{case f}) are counted by $P_1(m_1,m_2-1, m_3, m; q)$, $P_0(m_1,m_2-1, m_3, m; q)$ and $P_1(m_1,m_2-1, m_3, m-1; q)$, respectively. Therefore, we have:
\begin{align*} 
	P_1(m_1,m_2, m_3, m+1; q)= &\underbrace{q^{2m+1}}_{\substack{\text{for the deleted } \\ \text{pair $[m, m+1]$}}}\Big[P_1(m_1,m_2-1, m_3, m; q) \nonumber \\&+P_0(m_1,m_2-1, m_3, m; q) \nonumber
	\\&+P_1(m_1,m_2-1, m_3, m-1;q)\Big] 
\end{align*}
\noindent  We note that in the polynomials $P(m_1,m_2, m_3, m+1; q)$, we ignore the weights of the moveable singletons, but if they exist, they have the forms in the  streaks coloured blue. We recall that $P_0(m_1,m_2, m_3, m+1; q)$ is the generating function of the base partitions in which the largest pair is $[m, m]$ and $P_1(m_1,m_2, m_3,m+1; q)$ is the generating function of the base partitions in which the largest pair is $[m, m+1]$. $P_0(0, 0, 0, 1; q)$ is the generating function of the base partitions in which the largest pair is $[0, 0]$, but since we take $m_1,m_2, m_3=0$, we set $P_0(0, 0, 0, 1; q)=1$ which counts the base partition $\beta=1$ with only one moveable singleton.  
$P_1(0, 0, 0, 1; q)$ is the generating function of the base partitions in which the largest pair is $[0, 1]$, where $m_1,m_2, m_3=0$. Therefore, we set $P_1(0, 0, 0, 1; q)=0$. The other initial conditions in (\ref{initial conditions}) are obvious. 

\vspace{3mm}
\noindent The polynomials $P_0(m_1, m_2, m_3, m+1; q)$ and $P_1(m_1, m_2, m_3, m+1; q)$ satisfy the functional equations (\ref{functional equation of P0}) and  (\ref{functional equation of P1}) and the initial values of the polynomials are $0$ or $1$, (\ref{initial conditions}).  Therefore, $P(m_1,m_2, m_3, m+1; q)$'s are the polynomials of $q$ with evidently positive coefficients. 
\qed

\begin{cor} \label{h(n, m) and polynomials}
For $n, m \in \mathbb{N}$, let $h(n, m)$ denote the number of partitions of $n$ into $m$ parts such that each part appears at most twice. Then,
\begin{align*} 
	H(t;q)&=\sum_{m, n \geq 0}h(n, m)q^nt^m\\
	&=\sum_{\substack{m_1, m_2, m_3,\\  m,n_{12}\geq 0}}\frac{P(m_1, m_2, m_3, m+1; q)q^{mn_{12}+n_{12}^2}t^{2m_1+2m_2+5m_3+n_{12}}}{(q;q)_{n_{12}}(q^3;q^3)_{m_1+m_2+2m_3}}.
\end{align*} 
where $P(m_1, m_2, m_3, m+1; q)$'s are the polynomials of $q$ with evidently positive coefficients constructed in Lemma \ref{polynomials}.
\end{cor}

\begin{proof}
By Lemma \ref{polynomials}, we have:

\begin{align} \label{polynomial equation for weight of beta}
	&\sum_{\substack{n_1, n_2 \geq 0 \\ n_1=n_{11}+n_{12} \\ \beta}}\frac{q^{| \beta |}t^{2n_2+n_1}}{(q; q)_{n_{12}}(q^3; q^3)_{n_2}} \nonumber\\&=\sum_{\substack{m_1, m_2,m_3,\\ m,n_{12} \geq 0 }}\frac{P(m_1, m_2, m_3, m+1; q)q^{mn_{12}+n_{12}^2}t^{2m_1+2m_2+5m_3+n_{12}}}{(q;q)_{n_{12}}(q^3;q^3)_{m_1+m_2+2m_3}}
\end{align}

\noindent The result follows  by combining (\ref{polynomial equation for weight of beta})  and Theorem \ref{each part appears at most twice}.
\end{proof}

\noindent \emph{\textbf{Proof of Theorem \ref{2+2 not allowed-NEW}:}} 
We have:
\begin{align} \label{w/o weight of beta-first}
	&A(t;q;2)=\prod_{n= 1}^{\infty} \frac{(1+tq^{2n}+ t^2q^{4n})}{(1-tq^{2n-1})(1-t^2q^{4n})}.\frac{1}{(1-t)}  \quad \quad (\text{by Proposition \ref{evens appearing even times}})\nonumber\\
	&=\left( \prod_{n= 1}^{\infty} (1+tq^{2n}+ t^2q^{4n})\right)\prod_{n= 1}^{\infty}\frac{1}{(1-tq^{2n-1})(1-t^2q^{4n})}.\frac{1}{(1-t)} \nonumber\\
	&=H_0(t;q)\prod_{n= 1}^{\infty}\frac{1}{(1-tq^{2n-1})(1-t^2q^{4n})}.\frac{1}{(1-t)}  \quad \quad \quad \quad  \quad  \quad \quad  (\text{by (\ref{H0(t;q)}}))\nonumber\\
	&=H(t; q^2)\prod_{n= 1}^{\infty}\frac{1}{(1-tq^{2n-1})(1-t^2q^{4n})}.\frac{1}{(1-t)}\nonumber\\
	&=\sum_{\substack{m_1, m_2, m_3,\\  m,n_{12}\geq 0}}\frac{P(m_1, m_2, m_3, m+1; q^2)q^{2mn_{12}+2n_{12}^2}t^{2m_1+2m_2+5m_3+n_{12}}}{(q^2;q^2)_{n_{12}}(q^6;q^6)_{m_1+m_2+2m_3}}   \quad \quad  \text{(by Theorem \ref{h(n, m) and polynomials}})\nonumber
\end{align}
\begin{align} 
	&\times \frac{1}{(tq;q^2)_\infty(t^2q^4;q^4)_\infty}.\frac{1}{(1-t)}\nonumber\\
	&=\sum_{\substack{m_1, m_2,m_3,\\  m,n_{12}\geq 0}}\frac{P(m_1, m_2, m_3, m+1; q^2)q^{2mn_{12}+2n_{12}^2}t^{2m_1+2m_2+5m_3+n_{12}}}{(q^2;q^2)_{n_{12}}(q^6;q^6)_{m_1+m_2+2m_3}}\nonumber \quad \quad \text{(by (\ref{Euler 1}))}\\
	&\times \left(\sum_{i\geq 0} \frac{(tq)^i}{(q^2;q^2)_i} \right).\left( \sum_{j\geq 0}\frac{(t^2q^4)^j}{(q^4;q^4)_j}  \right). \left( \sum_{k\geq 0}t^k \right) \nonumber\\
	&=\sum_{\substack{m_1, m_2, m_3, \\ m,n_{12}, i, j, k\geq 0}} \frac{P(m_1, m_2, m_3, m+1; q^2)q^{2mn_{12}+2n_{12}^2+i+4j}t^{2m_1+2m_2+5m_3+n_{12}+i+2j+k}}{(q^2;q^2)_{n_{12}}(q^6;q^6)_{m_1+m_2+2m_3}(q^2;q^2)_i(q^4;q^4)_j} 
\end{align}

\noindent We now add the weight of the base partition $\beta$ to the exponent of $q$ in (\ref{w/o weight of beta-first}). The exponent of $t$ keeps track of the number of parts in the partition. Therefore, we need to consider the base partition $\beta$ with $2m_1+2m_2+5m_3+n_{12}+i+2j+k$ parts, namely $\beta=1+3+5+ \ldots +2(2m_1+2m_2+5m_3+n_{12}+i+2j+k)-1$. The weight of $\beta$ is $(2m_1+2m_2+5m_3+n_{12}+i+2j+k)^2$. It follows that
\begin{align}
	&\sum_{m,n \geq 0}kr_1(n,m)q^nt^m \nonumber\\
	&= \sum_{\substack{m_1, m_2, m_3, \\ m,n_{12}, i, j, k\geq 0}} \frac{P(m_1, m_2, m_3, m+1; q^2)q^{2mn_{12}+2n_{12}^2+i+4j}t^{2m_1+2m_2+5m_3+n_{12}+i+2j+k}}{(q^2;q^2)_{n_{12}}(q^6;q^6)_{m_1+m_2+2m_3}(q^2;q^2)_i(q^4;q^4)_j} \nonumber \\
	&\times \color{red}{q^{(2m_1+2m_2+5m_3+n_{12}+i+2j+k)^2}} \nonumber \\
	&= \sum_{\substack{m_1, m_2, m_3, \\ m,n_{12}, i, j, k\geq 0}} \frac{P(m_1, m_2, m_3, m+1; q^2)q^{2mn_{12}+2n_{12}^2+i+4j+(2m_1+2m_2+5m_3+n_{12}+i+2j+k)^2}}{(q^2;q^2)_{n_{12}}(q^6;q^6)_{m_1+m_2+2m_3}(q^2;q^2)_i(q^4;q^4)_j} \nonumber \\
	&\times t^{2m_1+2m_2+5m_3+n_{12}+i+2j+k} \nonumber
\end{align}
This proves (\ref{first new generating function}). The polynomials $P(m_1, m_2, m_3, m+1; q)$'s have evidently positive coefficients by Lemma \ref{polynomials}. Therefore, the generating function (\ref{first new generating function}) is an evidently positive series.
\qed

\begin{theorem} \label{1 not allowed-NEW}
	Consider the partitions satisfying the following conditions:
	
	\begin{enumerate}
		\item [(a)] No consecutive parts allowed. 
		\item [(b)] Odd parts do not repeat.
		\item [(c)] For a contiguous sub-partition $\lambda_i+\lambda_{i+1}+\lambda_{i+2}$, we have $| \lambda_i-\lambda_{i+2} | \geq 4 $ if $\lambda_{i+1}$ is even and appears more than once.
		\item [(d')] $1$ is not allowed to appear as a part.
	\end{enumerate}
	
	\noindent For $n, m \in \mathbb{N}$, let $kr_2(n, m)$ denote the number of partitions of n into m parts such that the partitions satisfy the conditions (a), (b), (c) and (d'). Then,
	\begin{align} \label{second new generating function}
		&\sum_{m,n \geq 0}kr_2(n,m)q^nt^m \nonumber\\
		&= \sum_{\substack{m_1, m_2, m_3, \\ m,n_{12}, i, j\geq 0}} \frac{P(m_1, m_2, m_3, m+1; q^2)q^{2mn_{12}+2n_{12}^2+i+(2m_1+2m_2+5m_3+n_{12}+i+2j)^2}}{(q^2;q^2)_{n_{12}}(q^6;q^6)_{m_1+m_2+2m_3}(q^2;q^2)_i(q^4;q^4)_j} \nonumber \\
		&\times t^{2m_1+2m_2+5m_3+n_{12}+i+2j} 
	\end{align}
	where $P(m_1, m_2, m_3, m+1; q)$'s are the polynomials of $q$  with evidently positive coefficients constructed in Lemma \ref{polynomials}. Moreover, the generating function (\ref{second new generating function}) is an evidently positive series. 
\end{theorem}

\begin{proof}
We have:
\begin{align} \label{w/o weight of beta-second}
	&B(t; q; 2)=\prod_{n = 1}^{\infty} \frac{(1+tq^{2n}+ t^2q^{4n})}{(1-tq^{2n-1})(1-t^2q^{4n})}.\frac{1}{(1-t^2)} \quad \quad (\text{by Proposition \ref{evens appearing even times-with even number of 0s}}) \nonumber \\
	&=\left( \prod_{n= 1}^{\infty} (1+tq^{2n}+ t^2q^{4n})\right)\prod_{n= 1}^{\infty}\frac{1}{(1-tq^{2n-1})(1-t^2q^{4n})}.\frac{1}{(1-t^2)} \nonumber\\
	&=H_0(t;q)\prod_{n= 1}^{\infty}\frac{1}{(1-tq^{2n-1})(1-t^2q^{4n})}.\frac{1}{(1-t^2)}  \quad \quad \quad \quad  \quad  \quad \quad  (\text{by (\ref{H0(t;q)}}))\nonumber \\
	&=H(t; q^2)\prod_{n= 1}^{\infty}\frac{1}{(1-tq^{2n-1})(1-t^2q^{4n})}.\frac{1}{(1-t^2)} \nonumber \\
	&=\sum_{\substack{m_1, m_2, m_3,\\  m,n_{12}\geq 0}}\frac{P(m_1, m_2, m_3, m+1; q^2)q^{2mn_{12}+2n_{12}^2}t^{2m_1+2m_2+5m_3+n_{12}}}{(q^2;q^2)_{n_{12}}(q^6;q^6)_{m_1+m_2+2m_3}}   \quad  (\text{by Theorem \ref{h(n, m) and polynomials}})\nonumber \\
	&\times \frac{1}{(tq;q^2)_\infty(t^2;q^4)_\infty} \nonumber \\
	&=\sum_{\substack{m_1, m_2,m_3,\\  m,n_{12}\geq 0}}\frac{P(m_1, m_2, m_3, m+1; q^2)q^{2mn_{12}+2n_{12}^2}t^{2m_1+2m_2+5m_3+n_{12}}}{(q^2;q^2)_{n_{12}}(q^6;q^6)_{m_1+m_2+2m_3}} \quad \quad \text{(by (\ref{Euler 1}))} \nonumber\\
	&\times \left(\sum_{i\geq 0} \frac{(tq)^i}{(q^2;q^2)_i} \right).\left( \sum_{j\geq 0}\frac{(t^2)^j}{(q^4;q^4)_j}  \right)\nonumber\\
	&=\sum_{\substack{m_1, m_2, m_3, \\ m,n_{12}, i, j\geq 0}} \frac{P(m_1, m_2, m_3, m+1; q^2)q^{2mn_{12}+2n_{12}^2+i}t^{2m_1+2m_2+5m_3+n_{12}+i+2j}}{(q^2;q^2)_{n_{12}}(q^6;q^6)_{m_1+m_2+2m_3}(q^2;q^2)_i(q^4;q^4)_j} 
\end{align}
\noindent We now add the weight of the base partition $\beta$ to the exponent of $q$ in (\ref{w/o weight of beta-second}). The exponent of $t$ keeps track of the number of parts in the partition. Therefore, we need to consider the base partition $\beta$ with $2m_1+2m_2+5m_3+n_{12}+i+2j$ parts, namely $\beta=1+3+5+ \ldots +2(2m_1+2m_2+5m_3+n_{12}+i+2j)-1$. The weight of $\beta$ is $(2m_1+2m_2+5m_3+n_{12}+i+2j)^2$. It follows that
\begin{align}
	&\sum_{m,n \geq 0}kr_2(n,m)q^nt^m \nonumber\\
	&= \sum_{\substack{m_1, m_2, m_3, \\ m,n_{12}, i, j\geq 0}} \frac{P(m_1, m_2, m_3, m+1; q^2)q^{2mn_{12}+2n_{12}^2+i}t^{2m_1+2m_2+5m_3+n_{12}+i+2j}}{(q^2;q^2)_{n_{12}}(q^6;q^6)_{m_1+m_2+2m_3}(q^2;q^2)_i(q^4;q^4)_j} \nonumber \\
	&\times \color{red}{q^{(2m_1+2m_2+5m_3+n_{12}+i+2j)^2}} \nonumber \\
	&= \sum_{\substack{m_1, m_2, m_3, \\ m,n_{12}, i, j\geq 0}} \frac{P(m_1, m_2, m_3, m+1; q^2)q^{2mn_{12}+2n_{12}^2+i+(2m_1+2m_2+5m_3+n_{12}+i+2j)^2}}{(q^2;q^2)_{n_{12}}(q^6;q^6)_{m_1+m_2+2m_3}(q^2;q^2)_i(q^4;q^4)_j} \nonumber \\
	&\times t^{2m_1+2m_2+5m_3+n_{12}+i+2j} \nonumber
\end{align}
This proves (\ref{second new generating function}). The polynomials $P(m_1, m_2, m_3, m+1; q)$'s have evidently positive coefficients by Lemma \ref{polynomials}. Therefore, the generating function (\ref{second new generating function}) is an evidently positive series.
\end{proof}

\begin{theorem}  \label{1, 2, 3 not allowed-NEW}
	Consider the partitions satisfying the following conditions:
	
	\begin{enumerate} 
		\item [(a)]  No consecutive parts allowed. 
		\item [(b)] Odd parts do not repeat.
		\item [(c)] For a contiguous sub-partition $\lambda_i+\lambda_{i+1}+\lambda_{i+2}$, we have $| \lambda_i-\lambda_{i+2} | \geq 4 $ if $\lambda_{i+1}$ is even and appears more than once.
		\item [($d''$)]$1, 2$ and $3$ are not allowed to appear as parts.
	\end{enumerate}
	
	\noindent For $n, m \in \mathbb{N}$, let $kr_3(n, m)$ denote the number of partitions of n into m parts such that the partitions satisfy the conditions (a), (b), (c) and (d''). Then,
	\begin{align} \label{third new generating function}
		&\sum_{m,n \geq 0}kr_3(n,m)q^nt^m \nonumber\\
		&= \sum_{\substack{m_1, m_2, m_3, \\ m,n_{12}, i, j\geq 0}} \frac{P(m_1, m_2, m_3, m+1;  q^2)}{(q^2;q^2)_{n_{12}}(q^6;q^6)_{m_1+m_2+2m_3}(q^2;q^2)_i(q^4;q^4)_j}  \\
		&\times q^{2mn_{12}+2n_{12}^2+3i+4j+4m_1+4m_2+10m_3+2n_{12}+(2m_1+2m_2+5m_3+n_{12}+i+2j)^2}t^{2m_1+2m_2+5m_3+n_{12}+i+2j} \nonumber
	\end{align}
	where $P(m_1, m_2, m_3, m+1; q)$'s  are the polynomials of $q$  with evidently positive coefficients constructed in Lemma \ref{polynomials}. Moreover, the generating function (\ref{third new generating function}) is an evidently positive series. 
\end{theorem}

\begin{proof}
In the proof of Theorem \ref{1, 2, 3 not allowed}, we showed that there is a $1$-$1$ correspondence between the partitions enumerated by $kr_2(n,m)$ and the partitions enumerated by $kr_3(n,m)$. To construct an evidently positive series as the generating function of the partitions enumerated by $kr_3(n,m)$, we let $t^m \mapsto t^mq^{2m}$ in (\ref{second new generating function}) and we get:
\begin{align}
	&\sum_{m,n \geq 0}kr_3(n,m)q^nt^m \nonumber\\
	&= \sum_{\substack{m_1, m_2, m_3, \\ m,n_{12}, i, j\geq 0}} \frac{P(m_1, m_2, m_3, m+1; q^2)q^{2mn_{12}+2n_{12}^2+i+(2m_1+2m_2+5m_3+n_{12}+i+2j)^2}}{(q^2;q^2)_{n_{12}}(q^6;q^6)_{m_1+m_2+2m_3}(q^2;q^2)_i(q^4;q^4)_j} \nonumber \\
	&\times (tq^2)^{2m_1+2m_2+5m_3+n_{12}+i+2j} \nonumber\\
	&= \sum_{\substack{m_1, m_2, m_3, \\ m,n_{12}, i, j\geq 0}} \frac{P(m_1, m_2, m_3, m+1; q^2)}{(q^2;q^2)_{n_{12}}(q^6;q^6)_{m_1+m_2+2m_3}(q^2;q^2)_i(q^4;q^4)_j} \nonumber \\
	&\times q^{2mn_{12}+2n_{12}^2+3i+4j+4m_1+4m_2+10m_3+2n_{12}+(2m_1+2m_2+5m_3+n_{12}+i+2j)^2}t^{2m_1+2m_2+5m_3+n_{12}+i+2j} \nonumber
\end{align}
This proves (\ref{third new generating function}). The polynomials $P(m_1, m_2, m_3, m+1; q)$'s have evidently positive coefficients by Lemma \ref{polynomials}. Therefore, the generating function (\ref{third new generating function}) is an evidently positive series.
\end{proof}

\section{Closed Formulas of the polynomials $P(m_1,m_2, m_3, m+1; q)$ for Some Special Cases}
\label{secSpecialCases} 


In this section, we list a few closed formulas 
for the polynomials $P(m_1,m_2, m_3, m+1; q)$ in some special cases.  
We also list some of the $P(m_1,m_2, m_3, m+1; q)$'s for small values of the parameters.  

\begin{prop}
 \begin{align}
 \label{eqPx00} 
  P(m_1, 0,0,m+1;q) & = 
    q^{2m_1^2 - 2 m m_1 + m^2 + m} \begin{bmatrix} m_1 \\ m - m_1 \end{bmatrix}_{q^2} \\ 
 \label{eqP0x0} 
  P(0, m_2,0,m+1;q) & = 
    q^{2m_2^2 + m_2 - 2 m m_2 + m^2 + m} \begin{bmatrix} m_2 - 1 \\ m - m_2 \end{bmatrix}_{q^2} \\ 
 \label{eqP00x} 
  P(0, 0 ,m_3, m+1;q) & = 
    \begin{cases} q^{10 m_3^2 + 23 m_3}, & \textrm{ if } m = 4 m_3 \\ 
        0, & \textrm{ otherwise } \end{cases} \\
 \label{eqPx0x} 
  P(m_1, 0 ,m_3, m_1 + 4 m_3 +1;q) & = 
    q^{ m_1^2 + m_1 + 5 m_1 m_3 + 10 m_3 + 23 m_3} \begin{bmatrix} m_1 + m_3 \\ m_1 \end{bmatrix}_{q^3} \\
 \label{eqP0xx} 
  P(0, m_2 ,m_3, m_2 + 4 m_3 +1;q) & = 
    q^{ m_2^2 + 2 m_2 + 5 m_2 m_3 + 10 m_3 + 23 m_3} \begin{bmatrix} m_2 + m_3 \\ m_2 \end{bmatrix}_{q^3}
 \end{align}
\end{prop}


\begin{proof}
 We remember that $P(m_1, 0,0,m+1;q) = P_0(m_1, 0,0,m+1;q) + P_1(m_1, 0,0,m+1;q)$.  
 \eqref{functional equation of P1} will yield a negative parameter for each term 
 on the right hand side of $P_1(m_1, 0,0,m+1;q)$.  
 By \eqref{initial conditions} each of these terms, 
 hence $P_1(m_1, 0,0,m+1;q)$ are zero.  
 We see that 
 \begin{align}
  \label{P becomes P_0}
  P(m_1, 0,0,m+1;q) = P_0(m_1, 0,0,m+1;q).  
 \end{align}
 Then, \eqref{functional equation of P0} applied to $P_0(m_1, 0,0,m+1;q)$ transforms to 
 \begin{align}
  \nonumber 
  P_0(m_1, 0,0,m+1;q) = q^{2m} P_0(m_1-1, 0,0,(m-1)+1;q) + q^{2m} P_0(m_1-1, 0,0,(m-2)+1;q).  
 \end{align}
 Using \eqref{P becomes P_0}, we obtain
 \begin{align}
  \nonumber
  & P(m_1, 0,0,m+1;q) \\
  \label{functional equation for Px00} 
  & = q^{2m} P(m_1-1, 0,0,(m-1)+1;q) + q^{2m} P(m_1-1, 0,0,(m-2)+1;q).  
 \end{align}
 The next step is a verification of the asserted formula \eqref{eqPx00}.  
 When we substitute that expression for each of the terms in \eqref{functional equation for Px00} 
 and simplify the like powers of $q$, the expression reduces to 
 \begin{align}
  \nonumber
  \begin{bmatrix} m_1 \\ m - m_1 \end{bmatrix}_{q^2}
  = q^{2m} \begin{bmatrix} m_1 - 1 \\ m - m_1 \end{bmatrix}_{q^2}
  + q^{2m} \begin{bmatrix} m_1 - 1 \\ m - m_1 - 1 \end{bmatrix}_{q^2}, 
 \end{align}
 which is one of the Pascal relations for the $q$-binomial coefficients~\cite{the-theory-of-partitions}.   
 
 Finally, when $m_1 = m \geq 0$, the lower index of the binomial coefficient becomes zero, 
 the upper index is non-negative, hence the $q$-binomial coefficient reduces to $1$.  
 To explain 
 \begin{align}
  \nonumber 
  P(m, 0,0,m+1;q) = q^{m^2 + m}, 
 \end{align}
 we observe that when a base partition contains $m$ pairs of repeating parts, 
 the smallest the next singleton can be is $m+1$.  
 Because, each pair will occupy an integer; 
 the smallest ones possible being $1, 2, 3, \ldots, m$.  
 We need to keep in mind that there are no blocks in this setting, 
 so no singleton is stuck between pairs.  
 
 The same argument shows that $P(m_1, 0,0,m+1;q) = 0$ when $m_1 > m$.  
 This takes care of all initial values, so the proof of \eqref{eqPx00} is complete.  
 
 It is necessary to point out that it is possible to make the proof of \eqref{eqPx00} fully combinatorial.  
 It is the standard argument making connections 
 to yet other types of moves of the pairs among themselves 
 and the combinatorial description of the $q$-binomial coefficients as 
 restricted partition generating functions~\cite{the-theory-of-partitions}.  
 
 The proof of \eqref{eqP0x0} is analogous.  
 The difference is the examination of part of the initial cases.  
 $P_1(0, 0, 0, 1; q) = \begin{bmatrix} -1 \\ 0 \end{bmatrix}_{q^2} = 0$ is already discussed 
 in the proof of Lemma \ref{polynomials}.  
 To argue that 
 \begin{align}
  \nonumber
  P_1(0, m, 0, m+1; q) = q^{m^2+2m} \begin{bmatrix} m-1 \\ 0 \end{bmatrix}_{q^2} = q^{m^2 + 2m}, 
 \end{align}
 for $m \geq 1$, we observe that the smallest singleton after $m$ pairs of consecutive parts 
 can be $m+1$.  
 
 The proof of \eqref{eqP00x} is simply lining up $m_3$ blocks 
 so that no further backward moves on the constituent pairs are possible, 
 and observing that the smallest singleton after those can be $4m_3 + 1$. 
 
 The proofs of \eqref{eqPx0x} and \eqref{eqP0xx} 
 are similar to those of \eqref{eqPx00} and \eqref{eqP0x0}.  
\end{proof}

Undoubtedly there are many more formulas like these, 
but a closed formula for the general case, 
i.e. for $P(m_1,m_2, m_3, m+1; q)$ without any restriction of the parameters,
looks out of reach at the moment.  
One reason seems to be the incompatibility of the bases $q^2$ and $q^3$ in the special cases 
in the above proposition.  

We tabulate or list explicit formulas for the polynomials 
for small values of the parameters in the appendix.  

\section{Conclusion}
\label{Conclusion}	

In Section \ref{Alternative Construction}, we gave an alternative construction for some generating functions which are first constructed by Kanade and Russell \cite{sk-mcr-staircases}. \noindent In \cite{sk-mcr-staircases}, Kanade and Russell presented the following identities as conjectures:

\begin{align} \label{first product}
	\sum_{i, j, k \geq 0} \frac{(-1)^kq^{(i+2j+3k)(i+2j+3k-1)+i+6j+3k^2+6k}}{(q;q)_i(q^4;q^4)_j(q^6;q^6)_k}&=\frac{1}{(q, q^4, q^6, q^8, q^{11};q^{12})_\infty} \\
	\sum_{i, j, k \geq 0} \frac{(-1)^kq^{(i+2j+3k)(i+2j+3k-1)+2i+2j+3k^2+6k}}{(q;q)_i(q^4;q^4)_j(q^6;q^6)_k}&=\frac{(q^6;q^{12})_\infty}{(q^2, q^3, q^4, q^8, q^9, q^{10} ;q^{12})_\infty} \label{second product} \\
	&=\frac{(q^6;q^{12})_\infty}{(q^2, q^3, q^4;q^6)_\infty} \nonumber\\
	\sum_{i, j, k \geq 0} \frac{(-1)^kq^{(i+2j+3k)(i+2j+3k-1)+4i+6j+3k^2+12k}}{(q;q)_i(q^4;q^4)_j(q^6;q^6)_k}&= \frac{1}{(q^4, q^5, q^6, q^7, q^8;q^{12})_\infty}. \label{third product}
\end{align}

\noindent In 2020, Bringmann, Jennings-Shaffer and Mahlburg gave the proofs of conjectured identities (\ref{first product}), (\ref{second product}) and (\ref{third product})  \cite{proofs and reductions of various conjectured partition identities of Kanade and Russell}. 

\noindent By using the method that we use for constructing generating functions (\ref{first generating function}), (\ref{second generating function}) and (\ref{third generating function}), one can construct hundreds of generating functions. By the constructing of hundreds of generating functions, we mean that we may change the conditions concerning unallowed initial parts and we can construct generating functions for those  partitions.  The key point is that the generating functions (\ref{first generating function}), (\ref{second generating function}) and (\ref{third generating function}) have a nice infinite product representations as in (\ref{first product}), (\ref{second product})and (\ref{third product}) (When we take $t=1$).

In Section \ref{Evidently Positive Series}, we constructed evidently positive series (\ref{first new generating function}), (\ref{second new generating function}) and (\ref{third new generating function}) as the generating functions of partitions. When we construct those evidently positive series, we first obtain an evidently positive series for a key infinite product, namely (\ref{H0(t;q)}). The idea in the construction for the key infinite product follows from \cite{AG-type series-capparelli}, \cite{AG-type series-kanade and russell} and \cite{AG-type series-schur}. 

In (\ref{first new generating function}), (\ref{second new generating function}) and (\ref{third new generating function}),  we had to consider the polynomials for the weight of the base partition $\beta$.  Because,  for a fixed number of pairs and singletons, there exist different base partitions. It would be better if we were able to write a closed formula for the polynomials $P(m_1,m_2, m_3, m+1; q)$ without any restriction of the parameters, but at least we know the functional equations concerning those polynomials. Better than that would be constructing a monomial of $q$ as the generating function of the base partitions $\beta$'s with a given number of pairs and singletons. It is a quite interesting question and it waits for an answer. 

\section*{Acknowledgements} 

The majority of the results in this paper are from the second author's Ph.D. dissertation.  
The first author acknowledges Bilim Akademisi BAGEP support with gratitute.  

\bibliographystyle{amsplain}

\newpage
\noindent
{\bf Appendix: }

For $m_3 = 0$: 

\rotatebox{90}{\small
\begin{tabular}{|c|c|c|}
 \hline
 $m_1 \backslash m_2$ & 1 & 2 \\ \hline 
 1 & $\begin{array}{l} P(1,1,0,\leq 2; q) = 0 \\ 
        P(1,1,0,3;q) = q^7 \\
        P(1,1,0,3;q) = q^{11} + q^{9} \\
        P(1,1,0, \geq 5; q) = 0\end{array}$ 
    & $\begin{array}{l} P(1,2,0,\leq 3; q) = 0 \\ 
        P(1,2,0,4;q) = q^{14} \\
        P(1,2,0,5;q) = q^{20} + q^{18} + q^{16} \\
        P(1,2,0,6;q) = q^{22} + q^{20} \\
        P(1,2,0, \geq 7; q) = 0\end{array}$ \\ \hline 
 2 & $\begin{array}{l} P(2,1,0,\leq 3; q) = 0 \\ 
        P(2,1,0,4;q) = q^{13} \\
        P(2,1,0,5;q) = q^{19}+2q^{17}+q^{15} \\
        P(2,1,0,6;q) = q^{23}+q^{21}+q^{19} \\
        P(2,1,0, \geq 7; q) = 0\end{array}$ 
    & $\begin{array}{l} P(2,2,0,\leq 4; q) = 0 \\ 
        P(2,2,0,5;q) = q^{22} \\
        P(2,2,0,6;q) = q^{30}+2q^{28}+2q^{26}+2q^{24} \\
        P(2,2,0,7;q) = q^{36}+q^{34}+4q^{32}+2q^{30}+2q^{28} \\
        P(2,2,0,8;q) = q^{38}+q^{36}+q^{34} \\
        P(2,2,0,\geq 9; q) = 0\end{array}$ \\ \hline 
 3 & $\begin{array}{l} P(3,1,0,\leq 4; q) = 0 \\ 
        P(3,1,0,5;q) = q^{21} \\
        P(3,1,0,6;q) = q^{29}+2q^{27}+2q^{25}+q^{23} \\
        P(3,1,0,7;q) = q^{35}+2q^{33}+3q^{31}+2q^{29} \\ 
            \hspace{3cm} +q^{27} \\
        P(3,1,0,8;q) = q^{39}+q^{37}+q^{35}+q^{33} \\
        P(3,1,0,\geq 9; q) = 0\end{array}$ 
    & $\begin{array}{l} P(3,2,0,\leq 5; q) = 0 \\ 
        P(3,2,0,6;q) = q^{32} \\
        P(3,2,0,7;q) = q^{42}+2q^{40}+3q^{38}+2q^{36}+2q^{34} \\
        P(3,2,0,8;q) = q^{50}+2q^{48}+5q^{46}+5q^{44}+6q^{42}+3q^{40} \\ 
            \hspace{3cm} +2q^{38} \\
        P(3,2,0,9;q) = q^{56}+q^{54}+4q^{52}+4q^{50}+5q^{48}+2q^{46} \\ 
            \hspace{3cm} +2q^{44} \\
        P(3,2,0,10;q) = q^{58}+q^{56}+q^{54}+q^{52} \\
        P(3,2,0,\geq 11; q) = 0\end{array}$ \\ \hline 
 4 & $\begin{array}{l} P(4,1,0,\leq 5; q) = 0 \\ 
        P(4,1,0,6;q) = q^{31} \\
        P(4,1,0,7;q) = q^{41}+2q^{39}+2q^{37}+2q^{35}+q^{33} \\
        P(4,1,0,8;q) = q^{49}+2q^{47}+4q^{45}+4q^{43}+4q^{41} \\ 
            \hspace{3cm} +2q^{39}+q^{37} \\
        P(4,1,0,9;q) = q^{55}+2q^{53}+3q^{51}+4q^{49}+3q^{47} \\ 
            \hspace{3cm} +2q^{45}+q^{43} \\
        P(4,1,0,10;q) = q^{59}+q^{57}+q^{55}+q^{53}+q^{51} \\
        P(4,1,0,\geq 11; q) = 0\end{array}$ 
    & $\begin{array}{l} P(4,2,0,\leq 6; q) = 0 \\ 
        P(4,2,0,7;q) = q^{44} \\
        P(4,2,0,8;q) = q^{56}+2q^{54}+3q^{52}+3q^{50}+2q^{48}+2q^{46} \\
        P(4,2,0,9;q) = q^{66}+2q^{64}+6q^{62}+6q^{60}+10q^{58}+7q^{56} \\ 
            \hspace{3cm} +7q^{54}+3q^{52}+2q^{50} \\
        P(4,2,0,10;q) = q^{74}+2q^{72}+5q^{70}+8q^{68}+10q^{66}+11q^{64} \\ 
            \hspace{3cm} +9q^{62}+7q^{60}+3q^{58}+2q^{56} \\
        P(4,2,0,11;q) = q^{80}+q^{78}+4q^{76}+4q^{74}+7q^{72}+5q^{70} \\ 
            \hspace{3cm} +5q^{68}+2q^{66}+2q^{64} \\
        P(4,2,0,12;q) = q^{82}+q^{80}+q^{78}+q^{76}+q^{74} \\
        P(4,2,0,\geq 13; q) = 0\end{array}$ \\ \hline 
\end{tabular}
}

\rotatebox{90}{\small
\begin{tabular}{|c|c|c|}
 \hline
 $m_1 \backslash m_2$ & 3 & 4 \\ \hline 
 1 & $\begin{array}{l} P(1,3,0,\leq 4; q) = 0 \\ 
        P(1,3,0,5;q) = q^{23} \\
        P(1,3,0,6;q) = q^{31}+q^{29}+2q^{27}+2q^{25} \\
        P(1,3,0,7;q) = q^{35}+2q^{33}+2q^{31}+2q^{29} \\
        P(1,3,0,8;q) = q^{37}+q^{35} \\
        P(1,3,0, \geq 9; q) = 0\end{array}$ 
    & $\begin{array}{l} P(1,4,0,\leq 5; q) = 0 \\ 
        P(1,4,0,6;q) = q^{34} \\
        P(1,4,0,7;q) = q^{44}+q^{42}+2 q^{40}+2 q^{38}+2 q^{36} \\
        P(1,4,0,8;q) = q^{50}+2 q^{48}+3 q^{46}+4 q^{44}+3 q^{42}+2 q^{40} \\
        P(1,4,0,9;q) = q^{54}+2 q^{52}+3 q^{50}+2 q^{48}+2 q^{46} \\
        P(1,4,0,10;q) = q^{56}+q^{54} \\
        P(1,4,0, \geq 11; q) = 0\end{array}$ \\ \hline 
 2 & $\begin{array}{l} P(2,3,0,\leq 5; q) = 0 \\ 
        P(2,3,0,6;q) = q^{33} \\
        P(2,3,0,7;q) = q^{43}+2 q^{41}+2 q^{39}+3 q^{37}+2 q^{35} \\
        P(2,3,0,8;q) = q^{51}+q^{49}+4 q^{47}+5 q^{45}+6 q^{43} \\ 
            \hspace{3cm} +3 q^{41}+3 q^{39} \\
        P(2,3,0,9;q) = q^{55}+2 q^{53}+4 q^{51}+4 q^{49}+3 q^{47} \\ 
            \hspace{3cm} +2 q^{45} \\
        P(2,3,0,10;q) = q^{57}+q^{55}+q^{53} \\
        P(2,3,0, \geq 11; q) = 0\end{array}$ 
    & $\begin{array}{l} P(2,4,0,\leq 6; q) = 0 \\ 
        P(2,4,0,7;q) = q^{46} \\
        P(2,4,0,8;q) = q^{58}+2 q^{56}+2 q^{54}+3 q^{52}+3 q^{50}+2 q^{48} \\
        P(2,4,0,9;q) = q^{68}+q^{66}+4 q^{64}+5 q^{62}+9 q^{60}+7 q^{58} \\ 
            \hspace{3cm} +8 q^{56}+4 q^{54}+3 q^{52} \\
        P(2,4,0,10;q) = q^{74}+2 q^{72}+5 q^{70}+7 q^{68}+10 q^{66}+9 q^{64} \\ 
            \hspace{3cm} +8 q^{62}+4 q^{60}+3 q^{58} \\
        P(2,4,0,11;q) = q^{78}+2 q^{76}+5 q^{74}+4 q^{72}+5 q^{70}+3 q^{68} \\ 
            \hspace{3cm} +2 q^{66} \\
        P(2,4,0,12;q) = q^{80}+q^{78}+q^{76} \\
        P(2,4,0, \geq 13; q) = 0\end{array}$ \\ \hline 
 3 & $\begin{array}{l} P(3,3,0,\leq 6; q) = 0 \\ 
        P(3,3,0,7;q) = q^{45} \\
        P(3,3,0,8;q) = q^{57}+2 q^{55}+3 q^{53}+3 q^{51}+3 q^{49}+2 q^{47} \\
        P(3,3,0,9;q) = q^{67}+2 q^{65}+5 q^{63}+7 q^{61}+10 q^{59}+9 q^{57} \\ 
            \hspace{3cm} +8 q^{55}+4 q^{53}+3 q^{51} \\
        P(3,3,0,10;q) = q^{75}+q^{73}+4 q^{71}+8 q^{69}+10 q^{67}+11 q^{65} \\ 
            \hspace{3cm} +12 q^{63}+8 q^{61}+4 q^{59}+3 q^{57} \\
        P(3,3,0,11;q) = q^{79}+2 q^{77}+4 q^{75}+6 q^{73}+6 q^{71}+5 q^{69} \\ 
            \hspace{3cm} +3 q^{67}+2 q^{65} \\
        P(3,3,0,12;q) = q^{81}+q^{79}+q^{77}+q^{75} \\
        P(3,3,0, 13 \geq ; q) = 0 \end{array}$ 
    & $\begin{array}{l} P(3,4,0,\leq 7 ; q) = 0 \\ 
        P(3,4,0,8;q) = q^{60} \\
        P(3,4,0,9;q) = q^{74}+2 q^{72}+3 q^{70}+3 q^{68}+4 q^{66}+3 q^{64}+2 q^{62} \\
        P(3,4,0,10;q) = q^{86}+2 q^{84}+5 q^{82}+7 q^{80}+12 q^{78}+13 q^{76} \\ 
            \hspace{3cm} +15 q^{74}+11 q^{72}+10 q^{70}+5 q^{68}+3 q^{66} \\
        P(3,4,0,11;q) = q^{96}+q^{94}+4 q^{92}+8 q^{90}+14 q^{88}+17 q^{86} \\ 
            \hspace{3cm} +25 q^{84}+22 q^{82}+23 q^{80}+17 q^{78}+12 q^{76} \\ 
            \hspace{3cm} +5 q^{74}+4 q^{72} \\
        P(3,4,0,12;q) = q^{102}+2 q^{100}+5 q^{98}+10 q^{96}+14 q^{94}+18 q^{92} \\ 
            \hspace{3cm} +19 q^{90}+19 q^{88}+14 q^{86}+10 q^{84}+5 q^{82}+3 q^{80} \\
        P(3,4,0,13;q) = q^{106}+2 q^{104}+5 q^{102}+6 q^{100}+8 q^{98}+6 q^{96} \\ 
            \hspace{3cm} +6 q^{94}+3 q^{92}+2 q^{90} \\
        P(3,4,0,14;q) = q^{108}+q^{106}+q^{104}+q^{102} \\
        P(3,4,0,  \geq 15; q) = 0 \end{array}$ \\ \hline 
\end{tabular}
}

\newpage 

For $m_3 = 1$: 

\rotatebox{90}{
\begin{tabular}{|c|c|c|c|}
 \hline
 $m_1 \backslash m_2$ & 0 & 1 & 2 \\ \hline 
 0 &  $\begin{array}{l} P(0, 0, 1, \leq 4; q) = 0 \\ 
    P(0, 0, 1, 5; q) = q^{13} \\ 
    P(0, 0, 1, \geq 6; q) = 0 \end{array}$
    & $\begin{array}{l} P(0, 1, 1, \leq 5; q) = 0 \\ 
    P(0, 1, 1, 6; q) = q^{24}+q^{21} \\ 
    P(0, 1, 1, 7; q) = q^{26} \\ 
    P(0, 1, 1, \geq 8; q) = 0 \end{array}$ 
    & $\begin{array}{l} P(0, 2, 1, \leq 6; q) = 0 \\ 
    P(0, 2, 1, 7; q) = q^{37}+q^{34}+q^{31} \\ 
    P(0, 2, 1, 8; q) = q^{41}+q^{39}+q^{38}+q^{36} \\ 
    P(0, 2, 1, 9; q) = q^{43} \\ 
    P(0, 2, 1, \geq 10; q) = 0 \end{array}$ \\ \hline 
 1 & $\begin{array}{l} P(1, 0, 1, \leq 5; q) = 0 \\ 
    P(1, 0, 1, 6; q) = q^{23}+q^{20} \\ 
    P(1, 0, 1, 7; q) = q^{27}+q^{25} \\ 
    P(1, 0, 1, \geq 8; q) = 0 \end{array}$ 
    & $\begin{array}{l} P(1, 1, 1, \leq 6; q) = 0 \\ 
    P(1, 1, 1, 7; q) = q^{36}+2 q^{33}+q^{30} \\ 
    P(1, 1, 1, 8; q) = q^{42}+2 q^{40} \\ 
        \hspace{10mm} +q^{39}+q^{38}+q^{37}+2 q^{35} \\ 
    P(1, 1, 1, 9; q) = q^{44}+q^{42} \\ 
    P(1, 1, 1, \geq 10 ; q) = 0 \end{array}$ 
    & $\begin{array}{l} P(1, 2, 1, \leq 7; q) = 0 \\ 
    P(1, 2, 1, 8; q) = q^{51}+2 q^{48}+2 q^{45}+q^{42} \\ 
    P(1, 2, 1, 9; q) = q^{59}+2 q^{57}+q^{56}+q^{55} \\ 
        \hspace{2mm} +2 q^{54}+3 q^{53} +3 q^{52} +q^{51}+2 q^{50} \\ 
        \hspace{2mm} +2 q^{49}+2 q^{47} \\ 
    P(1, 2, 1, 10; q) = q^{63}+3 q^{61} \\ 
        \hspace{2mm} +q^{60}+2 q^{59}+2 q^{58}+q^{57} +2 q^{56}+2 q^{54} \\ 
    P(1, 2, 1, 11; q) = q^{65}+q^{63} \\ 
    P(1, 2, 1, \geq 12; q) = 0 \end{array}$ \\ \hline 
 2 & $\begin{array}{l} P(2, 0, 1, \leq 6; q) = 0 \\ 
    P(2, 0, 1, 7; q) = q^{35} \\ 
        \hspace{20mm} +q^{32}+q^{29} \\ 
    P(2, 0, 1, 8; q) = q^{41}+q^{39} \\ 
        \hspace{2mm} +q^{38}+q^{37} +q^{36}+q^{34} \\ 
    P(2, 0, 1, 9; q) = q^{45} \\ 
        \hspace{15mm} +q^{43}+q^{41} \\ 
    P(2, 0, 1, \geq 10; q) = 0 \end{array}$ 
    & $\begin{array}{l} P(2, 1, 1, \leq 7; q) = 0 \\ 
    P(2, 1, 1, 8; q) = q^{50}+2 q^{47} \\ 
        \hspace{20mm} +2 q^{44}+q^{41} \\ 
    P(2, 1, 1, 9; q)= q^{58}+2 q^{56} \\ 
        +2 q^{55}+2 q^{54}+2 q^{53}+2 q^{52}+3 q^{51} \\ 
        \hspace{2mm} +2 q^{50} +2 q^{49}+q^{48}+2 q^{46} \\ 
    P(2, 1, 1, 10; q) = q^{64}+2 q^{62} \\ 
        \hspace{2mm} +q^{61}+3 q^{60} +q^{59}+2 q^{58} \\ 
        \hspace{2mm} +3 q^{57}+q^{56} +2 q^{55}+2 q^{53} \\ 
    P(2, 1, 1, 11; q) = q^{66}+q^{64}+q^{62} \\ 
    P(2, 1, 1, \geq 12; q) = 0 \end{array}$ 
    & $\begin{array}{l} P(2, 2, 1, \leq 8; q) = 0 \\ 
    P(2, 2, 1, 9; q) = q^{67}+2 q^{64} \\ 
        \hspace{2mm} +3 q^{61}+2 q^{58}+q^{55} \\ 
    P(2, 2, 1, 10; q) = q^{77}+2 q^{75}+2 q^{74} \\ 
        \hspace{2mm} +2 q^{73}+3 q^{72}+4 q^{71}+4 q^{70}+5 q^{69} \\ 
        \hspace{2mm} +4 q^{68}+3 q^{67}+5 q^{66}+4 q^{65}+2 q^{64} \\ 
        \hspace{2mm} +3 q^{63}+2 q^{62}+2 q^{60} \\ 
    P(2, 2, 1, 11; q) = q^{85}+2 q^{83}+q^{82} \\ 
        \hspace{2mm} +4 q^{81}+2 q^{80}+6 q^{79}+5 q^{78}+6 q^{77} \\ 
        \hspace{2mm} +6 q^{76}+6 q^{75}+7 q^{74}+6 q^{73}+4 q^{72} \\ 
        \hspace{2mm} +5 q^{71}+2 q^{70}+3 q^{69}+3 q^{67} \\ 
    P(2, 2, 1, 12; q) = q^{89}+3 q^{87}+q^{86} \\ 
        \hspace{2mm} +4 q^{85}+2 q^{84}+4 q^{83}+3 q^{82}+2 q^{81} \\ 
        \hspace{2mm} +4 q^{80}+q^{79}+3 q^{78}+2 q^{76} \\ 
    P(2, 2, 1, 13; q) = q^{91}+q^{89}+q^{87} \\ 
    P(2, 2, 1, \geq 14; q) = 0 \end{array}$ \\ \hline 
\end{tabular}
}

\newpage

For $m_3 = 2$: 

\begin{align}
 \nonumber P(0, 0, 2, \leq 8; q) & = 0  
 & P(0, 1, 2, \leq 9; q) & = 0 \\ 
 \nonumber P(0, 0, 2, 9; q) & = q^{46} 
 & P(0, 1, 2, 10; q) & = q^{65}+q^{62}+q^{59} \\ 
 \nonumber P(0, 0, 2, \geq 10; q) & = 0 
 & P(0, 1, 2, 11; q) & = q^{69}+q^{67} \\
 \nonumber & 
 & P(0, 1, 2, \geq 12; q) & = 0 
\end{align}
\begin{align}
 \nonumber P(0, 2, 2, \leq 10; q) & = 0 \\ 
 \nonumber P(0, 2, 2, 11; q) & = q^{86}+q^{83}+2 q^{80}+q^{77}+q^{74} \\ 
 \nonumber P(0, 2, 2, 12; q) & = q^{92}+q^{90}+q^{89}+q^{88}+q^{87}+q^{86}+q^{85}+q^{84}+q^{82} \\ 
 \nonumber P(0, 2, 2, 13; q) & = q^{96}+q^{94}+q^{92} \\ 
 \nonumber P(0, 2, 2, \geq 14; q) & = 0 
\end{align}
\begin{align}
 \nonumber & & P(1, 1, 2, \leq 10; q) & = 0 \\ 
 \nonumber P(1, 0, 2, \leq 9; q) & = 0 
 & P(1, 1, 2, 11; q) & = q^{85}+2 q^{82}+3 q^{79}+2 q^{76}+q^{73} \\ 
 \nonumber P(1, 0, 2, 10; q) & = q^{64}+q^{61}+q^{58} 
 & P(1, 1, 2, 12; q) & = q^{93}+2 q^{91}+q^{90}+2 q^{89} \\ 
 \nonumber P(1, 0, 2, 11; q) & = q^{70}+q^{68}+q^{66} 
 & & +2 q^{88}+2 q^{87}+q^{86}+q^{85}+2 \\ 
 \nonumber P(1, 0, 2, \geq 12; q) & = 0 
 & & q^{84}+2 q^{83}+2 q^{81} \\ 
 \nonumber & & P(1, 1, 2, 13; q) & = q^{97}+2 q^{95}+2 q^{93}+q^{91} \\ 
 \nonumber & & P(1, 1, 2, \geq 14 ; q) & = 0 
\end{align}
\begin{align}
 \nonumber P(1, 2, 2, \leq 11; q) & = 0 \\ 
 \nonumber P(1, 2, 2, 12; q) & = q^{108}+2 q^{105}+4 q^{102}+4 q^{99}+4 q^{96}+2 q^{93}+q^{90} \\ 
 \nonumber P(1, 2, 2, 13; q) & = q^{118}+2 q^{116}+q^{115}+2 q^{114}+3 q^{113}
    +3 q^{112}+2 q^{111}+6 q^{110} +4 q^{109}+4 q^{108} \\ 
 \nonumber & +4 q^{107}+6 q^{106}+3 q^{105}+5 q^{104}+2 q^{103}
    +2 q^{102}+3 q^{101}+2 q^{100}+2 q^{98} \\ 
 \nonumber P(1, 2, 2, 14; q) & = q^{124}+3 q^{122}+q^{121}+4 q^{120}+3 q^{119}
    +5 q^{118}+3 q^{117}+4 q^{116}+4 q^{115} \\ 
 \nonumber & +4 q^{114}+2 q^{113}+4 q^{112}+2 q^{111}+3 q^{110}+2 q^{108} \\ 
 \nonumber P(1, 2, 2, 15; q) & = q^{128}+2 q^{126}+3 q^{124}+2 q^{122}+q^{120} \\ 
 \nonumber P(1, 2, 2, \geq 16; q) & = 0 
\end{align}
\begin{align}
 \nonumber P(2, 0, 2, \leq 10; q) & = 0 \\ 
 \nonumber P(2, 0, 2, 11; q) & = q^{84}+q^{81}+2 q^{78}+q^{75}+q^{72} \\ 
 \nonumber P(2, 0, 2, 12; q) & = q^{92}+q^{90}+q^{89}+q^{88}+q^{87}
    +2 q^{86}+q^{85}+q^{84}+q^{83}+q^{82}+q^{80} \\ 
 \nonumber P(2, 0, 2, 13; q) & = q^{98}+q^{96}+2 q^{94}+q^{92}+q^{90} \\ 
 \nonumber P(2, 0, 2, \geq 14; q) & = 0 
\end{align}
\begin{align}
 \nonumber P(2, 1, 2, \leq 11; q) & = 0 \\ 
 \nonumber P(2, 1, 2, 12; q) & = q^{107}+2 q^{104}+4 q^{101}+4 q^{98}+4 q^{95}+2 q^{92}+q^{89} \\ 
 \nonumber P(2, 1, 2, 13; q) & = q^{117}+2 q^{115}+2 q^{114}+2 q^{113}+3 q^{112}+5 q^{111}
    +2 q^{110}+5 q^{109}+5 q^{108} \\ 
 \nonumber & +5 q^{107}+4 q^{106}+5 q^{105}+3 q^{104}
    +6 q^{103}+2 q^{102}+q^{101}+3 q^{100}+2 q^{99}+2 q^{97} \\ 
 \nonumber P(2, 1, 2, 14; q) & = q^{125}+2 q^{123}+q^{122}+4 q^{121}+2 q^{120}
    +5 q^{119}+3 q^{118}+5 q^{117}+4 q^{116}+5 q^{115} \\ 
 \nonumber & +4 q^{114}+5 q^{113}+2 q^{112}+4 q^{111}+2 q^{110}+3 q^{109}+2 q^{107} \\ 
 \nonumber P(2, 1, 2, 15; q) & = q^{129}+2 q^{127}+3 q^{125}+3 q^{123}+2 q^{121}+q^{119} \\ 
 \nonumber P(2, 1, 2, \geq 16; q) & =  0
\end{align}
\begin{align}
 \nonumber P(2, 2, 2, \leq 12; q) & = 0 \\ 
 \nonumber P(2, 2, 2, 13; q) & = q^{132}+2 q^{129}+5 q^{126}+6 q^{123}+8 q^{120}
    +6 q^{117}+5 q^{114}+2 q^{111}+q^{108} \\ 
 \nonumber P(2, 2, 2, 14; q) & = q^{144}+2 q^{142}+2 q^{141}+2 q^{140}+4 q^{139}
    +6 q^{138}+3 q^{137}+9 q^{136}+8 q^{135} \\ 
 \nonumber & +9 q^{134}+9 q^{133}+11 q^{132}+11 q^{131}+13 q^{130}+8 q^{129}+13 q^{128}+11 q^{127} \\
 \nonumber & +8 q^{126}+9 q^{125}+9 q^{124}+4 q^{123}+9 q^{122}+3 q^{121}+2 q^{120}+4 q^{119}+2 q^{118}+2 q^{116} \\ 
 \nonumber P(2, 2, 2, 15; q) & = q^{154}+2 q^{152}+q^{151}+5 q^{150}+3 q^{149}
    +7 q^{148}+6 q^{147}+12 q^{146}+8 q^{145}+15 q^{144} \\ 
 \nonumber & +12 q^{143}+18 q^{142}+13 q^{141}+20 q^{140}+13 q^{139}+21 q^{138}
    +14 q^{137}+17 q^{136} \\ 
 \nonumber & +10 q^{135}+14 q^{134}+8 q^{133}+11 q^{132}+4 q^{131}+6 q^{130}
    +4 q^{129}+4 q^{128}+3 q^{126} \\ 
 \nonumber P(2, 2, 2, 16; q) & = q^{160}+3 q^{158}+q^{157}+6 q^{156}+3 q^{155}
    +9 q^{154}+5 q^{153}+11 q^{152}+7 q^{151} \\ 
 \nonumber & +11 q^{150}+8 q^{149}+11 q^{148}+7 q^{147}+10 q^{146}+6 q^{145}+9 q^{144}+3 q^{143} \\ 
 \nonumber & +7 q^{142}+2 q^{141}+4 q^{140}+2 q^{138} \\ 
 \nonumber P(2, 2, 2, 17; q) & = q^{164}+2 q^{162}+4 q^{160}+4 q^{158}+4 q^{156}+2 q^{154}+q^{152} \\ 
 \nonumber P(2, 2, 2, \geq 18; q) & =  0
\end{align}

\end{document}